\documentclass[12pt]{amsart}
\usepackage{amscd,amsmath,amssymb,amsfonts,verbatim}
\usepackage[cmtip, all]{xy}

\newtheorem{thm}{Theorem}[section]
\newtheorem{prop}[thm]{Proposition}
\newtheorem{lem}[thm]{Lemma}
\newtheorem{cor}[thm]{Corollary}

\theoremstyle{definition}

\newtheorem{defn}[thm]{Definition}
\theoremstyle{remark}
\newtheorem{remk}[thm]{Remark}
\newtheorem{remks}[thm]{Remarks}

\newtheorem{exm}[thm]{Example}
\newtheorem{exms}[thm]{Examples}
\newtheorem{notat}[thm]{Notation}
\numberwithin{equation}{section}

{\hfill$\square$\end{defn}}
{\hfill$\square$\end{remk}}
{\hfill$\square$\end{remks}}
{\hfill$\square$\end{exm}}
{\hfill$\square$\end{exms}}
{\hfill$\square$\end{notat}}

\newcommand{\thmref}{Theorem~\ref}
\newcommand{\propref}{Proposition~\ref}
\newcommand{\corref}{Corollary~\ref}

\newcommand{\lemref}{Lemma~\ref}

\newcommand{\sC}{{\mathcal C}}

\newcommand{\sI}{{\mathcal I}}

\newcommand{\sL}{{\mathcal L}}

\newcommand{\sO}{{\mathcal O}}

\newcommand{\sU}{{\mathcal U}}

\newcommand{\sW}{{\mathcal W}}

\newcommand{\sY}{{\mathcal Y}}

\newcommand{\A}{{\mathbb A}}

\newcommand{\G}{{\mathbb G}}

\renewcommand{\P}{{\mathbb P}}

\newcommand{\Z}{{\mathbb Z}}

\newcommand{\CH}{{\rm CH}}

\newcommand{\inj}{\hookrightarrow}

\newcommand{\codim}{{\rm codim}}

\newcommand{\Spec}{{\rm Spec \,}}

\newcommand{\Sch}{{\operatorname{\mathbf{Sch}}}}

\newcommand{\Sm}{{\mathbf{Sm}}}

\newcommand{\ds}{{/\kern-3pt/}}

\newcommand{\Gr}{{\text{\rm Gr}}}

\newcommand{\Supp}{{\operatorname{Supp}}}

\newcommand{\Proj}{{\operatorname{Proj}}}

\newcommand{\un}{\underline}
\newcommand{\ov}{\overline}

\newcommand{\dgn}{{\operatorname{degn}}}
\renewcommand{\dim}{\text{\rm dim}}

\newcommand{\tuborg}{\left\{\begin{array}{ll}}
\newcommand{\sluttuborg}{\end{array}\right.}

\newcommand{\Sec}{{\rm Sec}}

\begin{document}
\title{A moving lemma for cycles with very ample modulus}
\author{Amalendu Krishna, Jinhyun Park}
\address{School of Mathematics, Tata Institute of Fundamental Research,  
1 Homi Bhabha Road, Colaba, Mumbai, India}
\email{amal@math.tifr.res.in}
\address{Department of Mathematical Sciences, KAIST, 291 Daehak-ro Yuseong-gu, 
Daejeon, 305-701, Republic of Korea (South)}
\email{jinhyun@mathsci.kaist.ac.kr; jinhyun@kaist.edu}

%Submitted to Scuola Normale on May 31, 2016. 


\maketitle

\begin{abstract}
We prove a moving lemma for higher Chow groups with modulus, in the sense of 
Binda-Kerz-Saito, of projective schemes when the modulus is given by a very 
ample divisor. This provides one of the first cases of moving lemmas for 
cycles with modulus, not covered by the additive higher Chow groups. 
We apply this to prove a contravariant functoriality of higher Chow groups 
with modulus. We use our moving techniques to show that the higher Chow
groups of a line bundle over a scheme, with the 0-section as the modulus,
vanishes.
\end{abstract}

\section{Introduction}\label{sec:Intro}

The moving lemma is one of the most important technical tools in 
dealing with algebraic cycles. For usual higher Chow groups, this
was established by S. Bloch (see \cite{Bl-1}, \cite{Bl-2}). 
In order to study the relative $K$-theory of schemes 
(relative to effective divisors) in terms of algebraic cycles, 
the theory of additive higher Chow groups 
(see \cite{BE2}, \cite{KL}, \cite{KP}, \cite{P1}) and 
cycles with modulus (see \cite{BS}, \cite{KS}) were recently introduced. 
But the lack of a moving lemma has been an annoying hindrance in the study of 
these additive higher Chow groups and the Chow groups with modulus.

A moving lemma for additive higher Chow groups of smooth projective schemes was proven in \cite{KP}. A similar moving lemma for the additive higher Chow groups of smooth affine schemes has been very recently established by W. Kai \cite{Kai}, along with some more general results after Nisnevich sheafifications. However, without such modifications, one does not yet know of the existence of a moving lemma for the higher Chow groups with modulus which do not arise from 
additive higher Chow groups. 

\subsection{Main results}
The goal of this paper is to address the moving lemma problem 
for the higher Chow groups
with modulus of projective schemes when the modulus divisor is very ample.
Our main result is the following. The necessary definitions are recalled in 
\S\ref{sec:Recall}.

\begin{thm}\label{thm:Main-Thm}
Let $X$ be an equidimensional reduced projective scheme of dimension 
$d \ge 1$ over a field $k$. 
Let $D \subsetneq X$ be a very ample effective Cartier divisor
such that $X \setminus D$ is smooth over $k$.
Let $\sW $ be a finite collection of locally closed subsets of $X$. 
Then, the inclusion $z^q_{\sW}(X|D, \bullet) \inj z^q(X|D, \bullet)$ is a 
quasi-isomorphism.
\end{thm}

Our first application of \thmref{thm:Main-Thm} is the following complete 
solution of the moving lemma
for cycles with arbitrary modulus on projective spaces.
The analogous question for cycles on 
affine spaces was solved by W. Kai \cite{Kai}.

\begin{cor}\label{cor:Proj-all}
Let $k$ be any field and $r \ge 1$ any integer.
Let $D \subset \P^r_k$ be any effective 
Cartier divisor. Let $\sW $ be a finite collection of locally closed subsets 
of $\P^r_k$. Then, the inclusion 
$z^q_{\sW}(\P^r_k|D, \bullet) \inj z^q(\P^r_k|D, \bullet)$ is a 
quasi-isomorphism.
\end{cor}

In the second application of \thmref{thm:Main-Thm}, we prove the following
contravariance property of the higher Chow groups with modulus.

\begin{thm}\label{thm:Intro-2}
Let $f: Y \to X$ be a morphism of equidimensional 
reduced quasi-projective schemes over a field $k$
where $X$ is projective over $k$. Let $D \subset X$ be a very ample 
effective Cartier divisor such that $X \setminus D$ is smooth over $k$. 
Suppose that $f^*(D)$ is a Cartier divisor on $Y$
(i.e., no minimal or embedded component of $Y$ maps into $D$). 
Then, there exists a map
\[
f^*: z^q(X|D, \bullet) \to z^q(Y|f^*(D), \bullet)
\]
in the derived category
of abelian groups. In particular, there is a pull-back
$f^*: \CH^q(X|D, p) \to \CH^q(Y|f^*(D), p)$ for every $p,q \ge 0$. 
\end{thm}

\begin{cor}\label{cor:Intro-2-0}
Let $r \ge 1$ be an integer and let $f: Y \to \P^r_k$ be a morphism of 
quasi-projective schemes over a field $k$. Let $D \subset \P^r_k$ be an 
effective Cartier divisor such that $f^*(D)$ is a Cartier divisor on $Y$. 
Then, there exists a pull-back $f^*: \CH^q(\P^r_k|D, p) \to \CH^q(Y|f^*(D), p)$ 
for every $p,q \ge 0$. 
\end{cor}

As the final application of our moving techniques, we prove the following
vanishing theorem for the higher Chow groups of a line bundle on a scheme
with the modulus given by the 0-section. This provides examples 
where the higher Chow groups of a variety with a modulus in an effective 
Cartier divisor are all zero. As one knows, this is not possible
for the ordinary higher Chow groups. 
This also gives an evidence in support of the
expectation that the higher Chow groups with modulus are the relative motivic
cohomology.

\begin{thm}\label{thm:Intro-3}
Let $X$ be a quasi-projective scheme over a field $k$ and let
$f: \sL \to X$ be a line bundle. Let $\iota:X \inj \sL$ denote the
0-section embedding. 
Then, the cycle complex $z_s(\sL|X, \bullet)$ is acyclic for all $s \in \Z$.
\end{thm}

\subsection{Outline of proofs}
We prove \thmref{thm:Main-Thm} by following the classical approach used by Bloch to prove his moving lemma for ordinary higher Chow groups of smooth projective schemes. We first prove the above theorem for projective spaces. The main difficulty here lies in constructing suitable homotopy varieties and to check their modulus condition. We solve this problem by using some blow-up techniques and our homotopy varieties are very different from the one used classically. 

To deal with the case of general projective schemes, we use the method of linear projections. However, we need to make more subtle choices of our linear 
subspaces than in the classical case due to the presence of the modulus.

We show later in this article how this method breaks down if we replace a very ample divisor by just an ample one. We show that the linear projection method can not be used in general to prove the moving lemma for Chow groups with modulus on either smooth affine or smooth projective schemes, if the modulus divisor is not very ample. This suggests that the general case of the moving lemma for Chow groups with modulus on smooth affine or projective schemes may be a very
challenging task.

\section{Recollection of cycles with modulus}\label{sec:Recall}
In this section, we recollect some needful definitions and notations associated with cycles with modulus. Let $k$ be a field and let $\Sch_k$ denote the category of quasi-projective schemes over $k$. Let $\Sm_k$ denote the full subcategory of $\Sch_k$ consisting of smooth schemes. 

\subsection{Notations}\label{sec:Note}
Set $\A^1_k:=\Spec k[t]$, $\P^1_k:=\Proj\, k[Y_0,Y_1]$ and let $y:=Y_0/Y_1$ be the coordinate on $\P^1_k$. We set $\square:= \A^1_k$ and $\ov{\square} := \P^1_k$. We use the coordinate system $(y_1, \cdots , y_{n})$ on $\ov{\square}^n$ with $y_i:=y\circ q_i$, where $q_i:  \ov{\square}^n \to \ov{\square}$ is the projection onto the $i$-th $\ov{\square}$. For $i=1,\ldots, n$, let $F^{\infty}_{n,i}$ be the Cartier divisor on $\ov{\square}^n$ defined by $\{y_i=\infty\}$. Let $F^{\infty}_{n}$ denote the Cartier divisor $\sum_{i=1} ^{n} F^{\infty}_{n,i}$ on $\ov{\square}^n$. A {\em face} of $\ov{\square}^n$ is a closed subscheme defined by a set of equations of the form $\{y_{i_1}=\epsilon_1, \ldots,  y_{i_s}=\epsilon_s| \ \epsilon_j\in\{0,1\}\}$. For $\epsilon=0, 1$,  and $i = 1, \cdots, n$, let $\iota_{n,i,\epsilon}: \ov{\square}^{n-1} \inj \ov{\square}^n$ be the inclusion
\begin{equation}\label{eqn:face}
\iota_{n,i,\epsilon}(y_1,\ldots, y_{n-1})=(y_1,\ldots, y_{i-1}, \epsilon,y_i,\ldots, y_{n-1}).
\end{equation}
A face of $\square^n$ is an intersection of $\square^n$ with a face of $\ov{\square}^n$.

\subsection{Cycles with modulus}\label{sect:CWM}
Let $X \in \Sch_k$. Recall (\cite[\S 2]{KP-1}) that for effective Cartier divisors $D_1$ and $D_2$ on $X$, we say $D_1 \leq D_2$ if $D_1 + D= D_2$ for some effective Cartier divisor $D$ on $X$. A \emph{modulus pair} or a \emph{scheme with an effective divisor} is a pair $(X,D)$, where $X \in \Sch_k$ and $D$ an effective Cartier divisor on $X$. A morphism $f: (Y, E) \to (X, D)$ of modulus pairs is a morphism $f: Y \to X$ in $\Sch_k$ such that $f^*(D)$ is defined as a Cartier divisor on $Y$ and $f^* (D) \leq E$. In particular, $f^{-1} (D) \subset E$. If $f: Y \to X$ is a morphism of $k$-schemes, and $(X, D)$ is a modulus pair such that $f^{-1} (D ) = \emptyset$, then $f: (Y , \emptyset) \to (X, D)$ is a morphism of modulus pairs.

\begin{defn}[{\cite{BS}, \cite{KS}}]\label{defn:modulus}
Let $(X,D)$ and $(\ov{Y},E)$ be two modulus pairs. Let $Y = \ov{Y} \setminus E$. Let $V \subset X \times Y$ be an integral closed subscheme with closure $\ov{V} \subset X \times \ov{Y}$. We say $V$ \emph{has modulus $D$ on $X \times Y$} (relative to $E$) if $\nu^*_V(D \times \ov{Y}) \le \nu^*_V(X \times E)$ on $\ov{V}^N$, where $\nu_V: \ov{V}^N \to \ov{V} \inj X \times \ov{Y}$ is the normalization followed by the closed immersion.
\end{defn}

\begin{defn}[{\cite{BS}, \cite{KS}}]\label{defn:partial complex}Let $(X,D)$ be a modulus pair. For $s \in \mathbb{Z}$ and $n \geq 0$, let $\un{z}_s (X|D, n)$ be the free abelian group on integral closed subschemes $V \subset X \times \square^n$ of dimension $s+n$ satisfying the following conditions:

\begin{enumerate}
\item (Face condition) for each face $F \subset \square^n$, $V$ intersects $X \times F$ properly. 
\item (Modulus condition) $V$ has modulus $D$ relative to $F_n^{\infty}$ on $X \times \square^n$. 
\end{enumerate}
\end{defn}

We usually drop the phrase ``relative to $F_n^{\infty}$'' for simplicity. A cycle in $\un{z}_s (X|D, n)$ is called an \emph{admissible cycle with modulus $D$}. 
The following containment lemma is from \cite[Proposition 2.4]{KP-1} (see also \cite[Lemma~2.1]{BS} and \cite[Proposition 2.4]{KP}).

\begin{prop}\label{prop:CL*}
Let $(X,D)$ and $(\ov{Y}, E)$ be modulus pairs and $Y= \ov{Y} \setminus E$. If $V \subset X \times Y$ is a closed subscheme with modulus $D$ relative to $E$, then any closed subscheme $W \subset V$ also has modulus $D$ relative to $E$.
\end{prop}

One checks using \propref{prop:CL*} that $(n \mapsto \un{z}_s(X|D,n))$ is a cubical abelian group. In particular, the groups  $\un{z}_s (X|D, n)$ form a complex with the boundary map $\partial= \sum_{i=1} ^n (-1)^i (\partial_i ^{0} - \partial_i ^1)$, where $\partial_i ^{\epsilon} = \iota_{n, i, \epsilon} ^*$. 

\begin{defn}[{\cite{BS}, \cite{KS}}]
The complex $(z_s(X|D, \bullet), \partial)$ is the nondegenerate complex associated to $(n\mapsto \un{z}_s(X|D,n))$, i.e., $z_s(X|D, n): = {\un{z}_s(X|D, n)}/{\un{z}_s(X|D, n)_{\dgn}}$. The homology $\CH_s(X|D, n): = {\rm H}_n (z_s(X|D, \bullet))$ for $n \geq 0$ is called \emph{higher Chow group of $X$ with modulus $D$}. If $X$ is equidimensional of dimension $d$, for $q \geq 0$, we write $\CH^q(X|D, n) = \CH_{d-q}(X|D, n)$.
\end{defn}

%If we take $X = Y \times \A^1_k$ and $D = Y \times (m+1)\{0\}$ for $m \ge 0$, then $\CH^q(X|D, n)$ is the additive higher Chow groups, written $\TH^q(Y,n+1; m)$.

The following is a generalization of \cite[Proposition~2.12]{KP-1} (see also 
\cite[Lemma~2.7]{BS}). The reader can check that the only requirement
in the proof of \cite[Proposition~2.12]{KP-1} is that the underlying map
be flat over the complement of the modulus divisor. This is because
of the fact that an admissible cycle lies completely over this complement.

\begin{lem}\label{lem:fpb}
Let $f: Y \to X$ be a morphism in $\Sch_k$. Let $D \subsetneq X$ be an
effective Cartier divisor. Assume that $f^*(D)$ is a Cartier divisor on $Y$
such that the map $f^{-1}(X \setminus D) \to X \setminus D$ is flat
of relative dimension $d$. Then, there is a pull-back map
$f^*: z_r(X|D, \bullet) \to z_{d+r}(Y|f^*(D), \bullet)$ such 
$(f \circ g)^* = g^* \circ f^*$.
\end{lem}

We often use the following result from \cite[Lemma 2.2]{KP-1}:

\begin{lem}\label{lem:cancel}
Let $f: Y \to X$ be a dominant map of normal integral $k$-schemes. Let $D$ be a Cartier divisor on $X$ such that the generic points of $\Supp(D)$ are contained in $f(Y)$. Suppose that $f^*(D) \ge 0$ on $Y$. Then, $D \ge 0$ on $X$.
\end{lem}

\begin{defn}\label{defn:complex for moving}
Let $\mathcal{W}$ be a finite set of locally closed subsets of $X$ and let $e: \mathcal{W} \to \mathbb{Z}_{\geq 0}$ be a set function. Let $\un{z}^q _{\mathcal{W}, e} (X|D, n) $ be the subgroup generated by integral cycles $Z \in \un{z}^q (X|D, n)$ such that for each $W \in \mathcal{W}$ and each face $F \subset \square^n$, we have $\codim _{W \times F} (Z \cap (W \times F)) \geq q - e(W)$.

They form a subcomplex $\un{z}^q _{\mathcal{W},e}(X|D, \bullet)$ of $\un{z}^q (X|D, \bullet)$. Modding out by degenerate cycles, we obtain the subcomplex $z^q _{\mathcal{W},e} (X|D, \bullet)\subset z^q(X|D, \bullet)$. We write $z^q _{\mathcal{W}} (X|D, \bullet) := z^q_{\mathcal{W}, 0} (X|D, \bullet)$. The number $e(W)$ is called the excess dimension of the intersection $Z \cap (W \times F)$.

Given a function $e: \sW \to \Z_{\ge 0}$, define $(e-1): \sW \to \Z_{\ge 0}$ by $(e-1)(W) = {\rm max}\{e(W)-1, 0\}$. This gives an inclusion $z^q _{\mathcal{W},e-1} (X|D, \bullet) \subset z^q _{\mathcal{W},e} (X|D, \bullet)$.
\end{defn}

\enlargethispage{25pt}

We also use the following from \cite[Proposition~4.3]{KP-2} in our proof of our moving lemma.

\begin{prop}[Spreading lemma]\label{prop:spread} Let $k \subset K$ be a purely transcendental extension. Let $(X,D)$ be a smooth quasi-projective $k$-scheme with an effective Cartier divisor, and let $\mathcal{W}$ be a finite collection of locally closed subsets of $X$. Let $(X_K, D_K)$ and $\mathcal{W}_K$ be the base changes via $\Spec (K) \to \Spec (k)$. Let ${\rm pr}_{K/k} : X_K \to X_k$ be the base change map. Then for every set function $e: \sW \to \Z_{\ge 0}$, the pull-back maps 
\begin{equation}\label{eqn:spread-0}
{\rm pr}_{K/k} ^* : \frac{z^q_{\sW,e} (X|D, \bullet) }{z_{\mathcal{W}}^q (X|D, \bullet)} \to \frac{z^q_{\sW_K,e} (X_K|D_K, \bullet )}{z_{\mathcal{W}_K}^q (X_K|D_K, \bullet)}
\end{equation}
and 
\begin{equation}\label{eqn:spread-1}
{\rm pr}_{K/k} ^* : \frac{z^q_{\sW,e} (X|D, \bullet) }{z_{\mathcal{W},e-1}^q (X|D, \bullet)} \to \frac{z^q_{\sW_K,e} (X_K|D_K, \bullet )}{z_{\mathcal{W}_K,e-1}^q (X_K|D_K, \bullet)}
\end{equation}
are injective on homology.
\end{prop}

We remark that \propref{prop:spread} is stated in \cite[Proposition~4.3]{KP-2} only for \eqref{eqn:spread-0} but the argument given there simultaneously proves \eqref{eqn:spread-1} as well.

\section{Moving lemma for projective spaces}\label{sec:PSM}

In this section, we prove our moving lemma for the modulus pair $(X,D)$, where $X$ is a projective space over $k$ and $D$ is a hyperplane in $X$. We use the following:

\begin{lem}[{\cite[Lemma 1.2]{Bl-1}}]\label{lem:Bloch-1}
Let $X \in \Sch_k$ and let $G$ be a connected algebraic group over $k$ acting on $X$. Let $A, B \subset X$ be closed subsets. Assume that the fibers of the action map $G \times A \to X$, given by $(g,a) \mapsto g\cdot a$, all have the same dimension and that this map is dominant.

Assume moreover that there is an overfield $k \inj K$ and a $K$-morphism $\psi: X_K \to G_K$. Let $\emptyset \neq U \subset X$ be open such that
for every $x \in U_K$, we have
\[
{\rm tr. deg}_k \left(\phi \circ \psi(x), \pi(x)\right) \ge \dim(G),
\]
where $\pi: X_K \to X$ and $\phi:G_K \to G$ are the base changes. Define $
\theta: X_K \to X_K
$
by $\theta(x) = \psi(x) \cdot x$ and assume that $\theta$ is an isomorphism. 
Then, the intersection $\theta(A_K \cap U_K) \cap B_K$ is proper.
\end{lem}

\begin{cor}\label{cor:Bloch-2}
Let $X \in \Sch_k$ and let $G$ be a connected algebraic group over $k$ acting transitively on $X$. Let $Y \in \Sch_k$ and let $\emptyset \neq A \subset X$ and $B \subset X \times Y$ be closed subsets. Let $G$ act on $X \times Y$ by $g \cdot (x,y) = (g\cdot x, y)$. 
%Assume that the fibers of the action map $G \times A \to X$, given by $(g,a) \mapsto g\cdot a$, all have same dimension and that this map is dominant. 

Let $K=k(G)$ and let $\phi: G_K \to G$ be the base change. Suppose $\psi : (X \times Y)_K \to G_K$ is a $K$-morphism and let $U \inj X \times Y$ be an open subset such that
\begin{enumerate}
\item the image of every point of $U_K$ under the composite map $(X \times Y)_K \xrightarrow{\psi} G_K \xrightarrow{\phi} G$ is the generic point of $G$. 
\item the map $\theta: (X \times Y)_K \to (X \times Y)_K$ given by $\theta(z) = \psi(z)\cdot z$, is an isomorphism.
\end{enumerate}
Then, the intersection $\theta((A \times Y)_K \cap U_K) \cap (B_K \cap U_K)$ is proper on $U_K$.
\end{cor}

We let $\A^r_k = \Spec(k[x_1, \cdots , x_r])$ and let $\P^r_k = {\rm Proj}(k[X_1, \cdots , X_r, X_0])$, where we set $x_i = {X_i}/{X_0}$ for $1 \le i \le r$. This yields an open immersion $j_0: \A^r_k \inj \P^r_k$. Let $H_{\infty} = \P^r_k \setminus \A^r_k$ be the hyperplane at infinity. We write the homogeneous coordinates of $\P^r_k$ as $(X_1; \cdots ; X_r; X_0)$. We fix this choice of coordinates of $\A^r_k$ and $\P^r_k$. Set $u = \stackrel{r}{\underset{i =1}\prod} x_i
\in k[x_1, \cdots , x_r]$. 

Let $K = k(\P^r_k)$ and consider the point $\eta = (u, \cdots , u) \in \P^r_K$ 
so that its image under the projection $\P^r_K \to \P^r_k$ is the generic 
point of $\P^r_k$. Let $U_{+} \inj \P^r_K \times \ov{\square}_K$ be the open subset $(\P^r_K \times \square_K) \cup (\A^r_K \times \ov{\square}_K)$ and set $\mathcal{Y} = H_{\infty} \times \{\infty\} = (\P^r_K \times \ov{\square}_K) \setminus U_{+}$. For $K$-schemes $X$ and $X'$, we write the product $X \times_K X'$ as $X \times X'$.

\begin{lem}\label{lem:generic-pt}
Let $\phi_{\eta}: \A^r_K \times \square_K \to \A^r_K$ denote the map $\phi_{\eta} (x, t) = x + \eta \cdot t$. Then, $\phi_{\eta}$ uniquely extends to a morphism $\phi_{\eta}|_{U_+}: U_{+} \to \P^r_K$ such that the following hold.
\begin{enumerate}
\item $U_{+}$ is the largest open subset of $\P^r_K \times \ov{\square}_K$ over which $\phi_{\eta}$ can be extended to a regular morphism. 
\item The extension of $\phi_{\eta}$ on $\P^r_K \times \square_K$ is a smooth morphism.
\item $(\phi_{\eta}|_{U_{+}})^{-1}(\A^r_K) = \A^r_K \times \square_K$.
\item $(\phi_{\eta}|_{U_{+}})^{-1}(H_{\infty}) = (\A^r_K \times \{\infty\}) +(H_{\infty} \times \square_K)$.
\end{enumerate}
\end{lem}

\begin{proof}
Define the rational map $\phi_{\eta}: \P^r_K \times \ov{\square}_K \dashrightarrow  \P^r_K$ by
\begin{equation}\label{eqn:generic-pt-rational}
\phi_{\eta}((X_1; \cdots ; X_r; X_0), (T_0; T_{1})) =  (T_1X_1 + uT_0X_0; 
\cdots ; T_1X_r + u T_0 X_0; T_1 X_0).
\end{equation} 

Note that $\phi_{\eta}((X_1; \cdots ; X_r; 1), (t; 1)) = 
(X_1 + ut; \cdots ; X_r + ut; 1)$ so that $\phi_{\eta}$ restricts to the given 
map on $\A^r_K \times \square_K$. One checks that (1), (3) and (4) hold from the shape of $\phi_{\eta}$ in \eqref{eqn:generic-pt-rational}.

To show (2), note that this map is the composite $\P^r_K \times \square_K \to \P^r_K \times \square_K \to \P^r_K$, where the first one is 
$\left((X_1; \cdots ; X_r; X_0), t\right) \mapsto 
\left((X_1 + utX_0; \cdots ; X_r + utX_0; X_0), t \right)$ 
and the second is the projection to $\P^r_K$ (which is smooth). Since the first map is an isomorphism, 
it follows that $\phi_{\eta}$ is smooth on $\P^r_K \times \square_K$.
\end{proof}

\begin{remk}\label{remk:eqn:generic-pt-flat}
The unique extension of $\phi_{\eta}$ to $U_{+}$ is not a flat morphism even though it is smooth on $\P^r_K \times \square_K$. If we set $V_i = \{(X_1; \cdots ; X_r; X_0)| X_i \neq 0\} \inj \P^r_K$ for $i = 1, \cdots , r$, then the map $\phi^{-1}_{\eta}(V_i) \to V_i$ is not flat because $\A^r_K \times \{0\}$ lies in one fiber but all other fibers have strictly smaller dimensions.
\end{remk} 

Our idea is to use the rational map 
$\phi_{\eta}: \P^r_K \times \ov{\square}_K \dashrightarrow  \P^r_K$
to generate a homotopy between an arbitrary admissible cycle
in $z^q(\P^r_k|H_{\infty}, \bullet)$ and a cycle in 
$z^q_{\sW,e} (\P^r_k|H_{\infty}, \bullet)$. In order to do so, we need to
extend $\phi_{\eta}$ to an honest morphism of schemes. We achieve this 
in the following results via a sequence of blow-ups.

\begin{lem}\label{lem:Blow-up}
Let $\pi: \Gamma \to \P^r_K \times \ov{\square}_K$ be the blow-up of $\P^r_K \times \ov{\square}_K$ along the closed subscheme $\mathcal{Y} = H_{\infty} \times \{\infty\}$. Then, there exists a closed point $P_{\infty} \in \pi^{-1}(\sY)$ and 
a regular 
map $\ov{\phi}_{\eta}: \Gamma_{+}:= \Gamma \setminus \{P_{\infty}\} \to \P^r_K$ 
such that  $\pi: \Gamma_+ \to \mathbb{P}^r _K \times \ov{\square} _K$ is 
surjective, and the diagram
\begin{equation}\label{eqn:Blow-up-0}
\xymatrix@C1pc{
\pi^{-1}(U_{+}) \ar@{^{(}->}[r]^>>>>{\ov{j}} \ar[d]_{\simeq} & \Gamma_{+} 
\ar@{->>}[d]_{\pi} 
\ar[dr]^{\ov{\phi}_{\eta}} & \\
U_{+} \ar@{^{(}->} [r]^>>>>{j} \ar@/_1.5pc/[rr] & \P^r_K \times 
\ov{\square}_K \ar@{.>}[r]^{\ \ \ \phi_{\eta}} & \P^r_K}
\end{equation}
commutes.
\end{lem}

\begin{proof}
Let $U_i \subsetneq \P^r_K$ be the open set $\{X_i \neq 0\}$ for 
$0 \le i \le r$. One checks by a direct local calculation the blow-up $\Gamma$ 
has the following description. Over $U_i$, it is defined by 
\begin{equation}\label{eqn:Blow-up-0-0}
\pi^{-1}(U_i) = \{\left((X_1; \cdots ; X_r; X_0), (T_0; T_{1}),
(Y_{1,i}; Y_{0,i})\right) \in U_i \times \ov{\square}_K \times \P^1_K| 
X_0T_0Y_{0,i} = X_iT_1Y_{1,i}\}
\end{equation}
and these blow-ups glue along their intersections
to make up $\Gamma$ via the change of coordinate
${Y_{0,i}}/{Y_{0,j}} = ({X_i}/{X_j})({Y_{1,i}}/{Y_{1,j}})$ over
$U_i \cap U_j$.
The blow-up map
$\pi: \pi^{-1}(U_i) \to U_i \times \ov{\square}_K$ is the composite 
$\pi^{-1}(U_i) \inj U_i \times \ov{\square}_K \times \P^1_K \to U_i \times 
\ov{\square}_K$.

We now define a rational map 
$\ov{\phi}^i_{\eta}: \pi^{-1}(U_i) \dashrightarrow  \P^r_K$ by
\begin{equation}\label{eqn:Blow-up-1}
\ov{\phi}_{\eta}\left((X_1; \cdots ; X_r; X_0), (T_0; T_{1}), 
(Y_{1,i}; Y_{0,i})\right) = \hspace*{5cm} 
\end{equation}
\[
\hspace*{5cm} 
\left(Y_{0,i}X_1 + uX_iY_{1,i}; \cdots ; Y_{0,i}X_r + uX_iY_{1,i}; 
Y_{0,i}X_0\right).
\]

The blow-up $\Gamma$ is glued along $U_i \cap U_j$ via the automorphism
$\psi_{i,j}: \pi^{-1}(U_i \cap U_j) \xrightarrow{\simeq} 
\pi^{-1}(U_i \cap U_j)$:
\[
\psi_{i,j}\left((X_1; \cdots ; X_r; X_0), (T_0; T_{1}), 
(Y_{1,i}; Y_{0,i})\right) = \hspace*{5cm} 
\]
\[
\hspace*{5cm} \left((X_1; \cdots ; X_r; X_0), (T_0; T_{1}), 
(X_i X^{-1}_jY_{1,i}; X_j X^{-1}_iY_{0,i})\right).
\]

It is clear from this isomorphism that $\psi_{i,j}(Y_{l,i} \neq 0) =
(Y_{l,j} \neq 0)$ for $l = 0,1$. Over $(Y_{0,i} \neq 0)$, we can let
$Y_{0,i} = Y_{0,j} = 1, Y_{1,i} = y_i$ and $Y_{1,j} = y_j$. Over this open
subset of $\pi^{-1}(U_i \cap U_j)$, we
get
\begin{equation}\label{eqn:Blow-up-1-0}
\ov{\phi}^j_{\eta} \circ \psi_{i,j} 
\left((X_1; \cdots ; X_r; X_0), (T_0; T_{1}), y_i\right)  = \hspace*{5cm} 
\end{equation}
\[
\begin{array}{lll}
& = & \ov{\phi}^j_{\eta}
\left((X_1; \cdots ; X_r; X_0), (T_0; T_{1}), X_i X^{-1}_j y_i\right) \\
& = & \left(X_ 1 + uX_j X_i X^{-1}_j y_i; \cdots ;
X_r + u X_j X_i X^{-1}_j y_i; X_0\right) \\
& = & \left( X_1 + u X_i y_i; \cdots ; X_r + u X_i y_i; X_0\right) \\
& = & \ov{\phi}^i_{\eta}
\left((X_1; \cdots ; X_r; X_0), (T_0; T_{1}), y_i\right).
\end{array}
\]

Over the intersection of $\pi^{-1}(U_i \cap U_j)$ with the open
subset $(Y_{1,i} \neq 0)$, we have
\begin{equation}\label{eqn:Blow-up-1-1}
\ov{\phi}^j_{\eta} \circ \psi_{i,j} 
\left((X_1; \cdots ; X_r; X_0), (T_0; T_{1}), y_i\right)  = 
\hspace*{5cm} 
\end{equation}
\[
\begin{array}{lll}
& = & \ov{\phi}^j_{\eta}
\left((X_1; \cdots ; X_r; X_0), (T_0; T_{1}), X_j X^{-1}_i y_i\right) \\
& = & \left(X_j X^{-1}_i X_1 y_i + uX_j; \cdots X_j X^{-1}_i X_r y_i + uX_j;
X^{-1}_i X_j X_0 y_i\right) \\
& = & \left(X_1 X_j y_i + uX_i X_j; \cdots ; X_r X_j y_i + u X_i X_j;
X_j X_0y_i\right) \\
& = & \left( X_1 y_i + u X_i; \cdots ; X_r y_i + u X_i; X_0 y_i\right) \\
& = & \ov{\phi}^i_{\eta}  
\left((X_1; \cdots ; X_r; X_0), (T_0; T_{1}), y_i\right).
\end{array}
\]

It follows from ~\eqref{eqn:Blow-up-1-0} and  ~\eqref{eqn:Blow-up-1-1} that
$\ov{\phi}^j_{\eta}$'s glue together to yield a rational
map $\ov{\phi}_{\eta}: \Gamma \dashrightarrow \P^r_K$ such that
$\ov{\phi}_{\eta}|_{\pi^{-1}(U_i)} = \ov{\phi}^j_{\eta}$ for $0 \le i \le r$.

We next show the commutativity of ~\eqref{eqn:Blow-up-0}.
The left square of \eqref{eqn:Blow-up-0} commutes by construction. 
We thus have to show that 
$\ov{\phi}_{\eta} \circ \ov{j} = \phi_{\eta} \circ \pi$, i.e., 
the trapezoid in \eqref{eqn:Blow-up-0} commutes. It suffices to show
this over each open subset $(U_i \times \ov{\square}_K) \cap U_{+}$.
If $P= \left((X_1; \cdots; X_r ; X_0 ), (T_0;T_1), (Y_{1,i};Y_{0,i})\right)
\in \pi^{-1} (U_+)$, we have $\pi (P) = \left((X_1; \cdots; X_r;X_0), 
(T_0;T_1)\right)$ such that either $T_1 \not = 0$ or $X_0 \not = 0$. 

Suppose first that $T_1 \not = 0$. Then, we can take $T_1 = 1$ and $T_0 = t$. 
In this case, we must have $Y_{0,i} \not = 0$ so that we can assume 
$Y_{0,i} =1$. Thus, the equation $X_0 T_0 Y_{0,i} = X_iT_1Y_{1,i}$ becomes 
$Y_{1,i} = t X_0 X^{-1}_i$.
This yields $\ov{\phi}^i_{\eta} \circ \ov{j}(P)
= \left(X_1 + utX_0; \cdots ; X_r + utX_0; X_0\right)$ by 
~\eqref{eqn:Blow-up-1} and $\phi_{\eta} \circ \pi (P) =
\left(X_1 + utX_0; \cdots ; X_r + utX_0; X_0\right)$ by 
~\eqref{eqn:generic-pt-rational}. 

Suppose next that $X_0 \neq 0$. Since the case $T_1 \not =0$ was already 
considered, we may suppose $T_0 \not = 0$. Thus, we may take $T_0 = 1$ and 
$T_1 = t$. In this case, we must have $Y_{1,i} \not = 0$, so that we may take 
$Y_{1,i} = 1$. Thus, the equation $X_0 T_0 Y_{0,i} = X_iT_1 Y_{1,i}$ becomes 
$Y_{0,i} = tX_iX^{-1}_0$. This yields $\ov{\phi}^i_{\eta} \circ \ov{j}(P) = 
\left(tX_1X_i+ uX_0X_i; \cdots ; tX_rX_i+ uX_0X_i; tX_iX_0\right)
= \left(tX_1+ X_0; \cdots ; tX_r + X_0; tX_0\right)$ by ~\eqref{eqn:Blow-up-1}.
On the other hand, $\phi_{\eta} \circ \pi (P) 
= \left(tX_1+ uX_0; \cdots ; tX_r + X_0; tX_0\right)$ by 
~\eqref{eqn:generic-pt-rational}.
We have thus shown that $\ov{\phi}_{\eta} \circ \ov{j}(P) =
\phi_{\eta} \circ \pi (P)$ for $P \in \pi^{-1}(U_+)$. 

We now show that $\ov{\phi}_{\eta}$ is regular on $\Gamma \setminus 
\{P_{\infty}\}$, where $P_{\infty} \in 
\left(\stackrel{r}{\underset{i = 1}\cap} \pi^{-1}(U_i)\right)$
is the closed point 
$\left((1;\cdots; 1;0), (1;0),(1;-u)\right)$ in the coordinates of
$\pi^{-1}(U_i)$.
Let $Q = \left((X_1; \cdots ; X_r; X_0), (T_0; T_{1}), (Y_{1,i}; Y_{0,i})\right) 
\in \pi^{-1}(U_i)$ be a point so that $X_0T_0Y_{0,i} = X_iT_1Y_{1,i}$.
Then $\ov{\phi}_{\eta}(Q)$ is not defined if and only if all its coordinates 
are zero, i.e.,
\begin{equation}\label{eqn:Blow-up-2}
Y_{0,i}X_j + uX_iY_{1,i} = 0, \ \ \mbox{ for all } \ 1 \le j \le r, \ \ 
\mbox{and} \ \ Y_{0,i}X_0 = 0.
\end{equation}

If $Y_{0,i} = 0$, then $uX_i Y_{1,i} = 0$ for $1 \le i \le r$. 
But $u \in K^{\times}$ and $Q \in \pi^{-1}(U_i)$ imply that $Y_{1,i} =0$, 
which can not happen since $(Y_{1,i};Y_{0,i}) \in \P^1_K$. 
So, $Y_{0,i} \not = 0$ and we must have $X_0 = 0$. 
Since $X_i \neq 0$, we can assume $X_i = 1$. Since $X_0 = 0$, we also have 
$T_1Y_{1,i} = 0$, so that either $Y_{1,i} = 0$ or $T_1 = 0$.
If $Y_{1,i} = 0$, then it follows from ~\eqref{eqn:Blow-up-2} 
that $Y_{0,i} = -u Y_{1,i} = 0$, which again 
is absurd because $(Y_{1,i};Y_{0,i}) \in \P^1_K$. 
So, $Y_{1,i} \not = 0$, and $T_1 = 0$. We may assume $Y_{1,i} = 1$.
 Combining this with \eqref{eqn:Blow-up-2}, we thus have
\begin{equation}\label{eqn:Blow-up-3}
Y_{0,i} = -u, \ \ Y_{0,i}X_j + u  = 0 \ \  \ \mbox{ for all }\ 1 \le j \neq i 
\le r \ \ 
\mbox{and} \ \ X_0 = T_1 = 0.
\end{equation}
We conclude that $\ov{\phi}_{\eta}(Q)$ is not defined if and only if
$Q = \left((1; \cdots ; 1; 0), (1;0), (1; -u)\right)$.
This proves the regularity of $\ov{\phi}_{\eta}$ on $\Gamma \setminus 
\{P_{\infty}\}$.
Since $P_{\infty} \in \pi^{-1}(\sY)$ and since each fiber of $\pi$ over
$\sY$ is 1-dimensional, we conclude that the map
$(\Gamma \setminus \{P_{\infty}\}) \to \P^r_K \times \ov{\square}_K$ is 
surjective. This finishes the proof of the lemma.
\end{proof}

\begin{remk}\label{remk:linear-system}
The reader can check that the map $\phi_{\eta}: \P^r_K \times \ov{\square}_K 
\dashrightarrow \P^r_K$ is the one defined by the linear system
generated by the global sections
$S = \{T_1X_i + uT_0X_0\}_{1 \le i \le r} \cup \{T_1X_0\}$
of the line bundle $\sO(1,1)$.
The sheaf of ideals $\sI_{\infty}$  on 
$\P^r_K \times \ov{\square}_K$ defining $\sY$ is
generated by $\{X_iT_1, X_0T_0| 0 \le i \le r\}$.
Moreover,
$\ov{\phi}_{\eta}: \Gamma \dashrightarrow \P^r_K$ 
is the rational map defined by the linear system
generated by the global sections $\pi^*(S)$ of the line bundle 
$\pi^{*}\sI_{\infty}$.
\end{remk}

Let $\pi:\Gamma \to \P^r_K \times \ov{\square}_K$ be the blow-up map as in 
\lemref{lem:Blow-up} and let $E = \pi^*(\sY)$ denote the exceptional
divisor for this blow-up. Note that the map $\pi: E \to \sY \simeq H_{\infty}$
is the $\P^1_K$-bundle associated to the vector bundle $\sO(1) \oplus \sO$.  

Since $H_{\infty} \times \ov{\square}_K$ and 
$\P^r_K \times \{\infty\}$ are smooth schemes, and $\mathcal{Y}$ is a smooth 
divisor inside these schemes, note that 
${\rm Bl}_{\mathcal{Y}}(H_{\infty} \times \ov{\square}_K) \to 
H_{\infty} \times \ov{\square}_K$ and 
${\rm Bl}_{\mathcal{Y}}(\P^r_K \times \{\infty\}) \to 
\P^r_K \times \{\infty\}$ are isomorphisms.

\begin{lem}\label{lem:Blow-up-trivial}
Let $\pi: \Gamma \to \P^r_K \times \ov{\square}_K$ be as in 
\lemref{lem:Blow-up}. Then, we have the following.
\begin{enumerate}
\item ${\rm Bl}_{\mathcal{Y}}(H_{\infty} \times \ov{\square}_K) \cap \{P_{\infty}\} 
= \emptyset = {\rm Bl}_{\mathcal{Y}}(\P^r_K \times \{\infty\}) \cap 
\{P_{\infty}\}$.
\item ${\rm Bl}_{\mathcal{Y}}(H_{\infty} \times \ov{\square}_K) \cap 
{\rm Bl}_{\mathcal{Y}}(\P^r_K \times \{\infty\}) = \emptyset$ inside $\Gamma$.
\item $\pi^*(H_{\infty} \times \ov{\square}_K) = 
(H_{\infty} \times \ov{\square}_K) + E$ and 
$\pi^*(\P^r_K \times \{\infty\}) = (\P^r_K \times \{\infty\}) + E$ in the 
group ${\rm Div}(\Gamma)$ of Weil divisors.
\end{enumerate}
\end{lem}

\begin{proof}
It suffices to verify each statement of the lemma over an open
subset $\pi^{-1}(U_i)$ with $0 \le i \le r$. On the other hand, 
~\eqref{eqn:Blow-up-0-0} shows that over $U_i$, we have
${\rm Bl}_{\mathcal{Y}}(H_{\infty} \times \ov{\square}_K) =
\{\left((X_1; \cdots ; X_r;0), (T_0;T_1), (Y_{1,i};Y_{0,i})\right) \in 
\P^r_K \times \ov{\square}_K \times \P^1_K|Y_{1,i} = 0\} = 
H_{\infty} \times \ov{\square}_K \times \{0\}$.
Similarly, we have 
${\rm Bl}_{\mathcal{Y}}(\P^r_K \times \{\infty\}) = 
\{\left((X_1; \cdots ; X_r;X_0), (1;0),(Y_{1,i};Y_{0,i})\right)$ \\
$\in \P^r_K 
\times \ov{\square}_K \times \P^1_K|Y_{0,i} = 0\} = 
\P^r_K \times \{\infty\} \times \{\infty\}$.
Since $P_{\infty}$ does not map to $\{0,\infty\} \subset \mathbb{P}^1 _K$ under 
the projection $\pi^{-1}(U_i) \to \P^1_K$ for
any $0 \le i \le r$, we get (1). 
The parts (2) and (3) of the lemma are immediate.
\end{proof}

Let $\Gamma_1 \inj \Gamma_{+} \times \P^r_K$ denote the graph of $\ov{\phi}_{\eta}$ and let $\ov{\Gamma}_1 \inj \Gamma \times \P^r_K$ be its closure. Let $\pi^N :\ov{\Gamma}_1 ^N \to  \ov{\Gamma}_1 \hookrightarrow \Gamma \times \mathbb{P}^r _K$ be the normalization composed with the inclusion, and let $\pi_1:= pr_1 \circ \pi^N$, $\pi_2:= pr_2 \circ \pi^N$, where $pr_1, pr_2$ are the projections from $\Gamma \times \mathbb{P}_K ^r$ to $\Gamma$ and $\mathbb{P}_K ^r$, respectively. Here, $\pi^N$ is finite and $\pi_1$ is projective with $\pi^{-1}_1(\Gamma_{+}) \xrightarrow{\simeq} \Gamma_{+}$ such that $\pi_2|_{\Gamma_{+}} = \ov{\phi}_{\eta}$.

Since $\pi_1$ is a birational projective morphism and 
$\Gamma$ is smooth, it follows from 
\cite[Theorem II-7.17, p.166, Exercise II-7.11(c), p.171]{Hart} 
that there is a closed subscheme $Z \inj \Gamma$ such that 
$Z_{\rm red} = \{P_{\infty}\}$ and $\ov{\Gamma}^N_1 = {\rm Bl}_{Z}(\Gamma)$. 
Let $F \inj \ov{\Gamma}^N_1$ denote the exceptional divisor for this blow-up 
so that $F_{\rm red} = \pi^{-1}_1(P_{\infty})$. Let $E_1 \inj \ov{\Gamma}^N_1$ 
denote the strict transform of $E$ 
under $\pi_1$ so that $\pi^*_1(E) = E_1 + F$.  

Letting $\delta := \pi \circ \pi_1: \ov{\Gamma}^N_1 \to 
\P^r_K \times \ov{\square}_K$ and $E' := \pi^*_1(E) = E_1 + F$, 
a combination of Lemmas~\ref{lem:Blow-up}, ~\ref{lem:Blow-up-trivial} and the 
above construction proves the following.

\begin{lem}\label{lem:Blow-up-Final}
There exists a commutative diagram
\begin{equation}\label{eqn:Blow-up-more}
\xymatrix@C1pc{
\delta^{-1}(U_{+}) \ar@{^{(}->}[r]^>>>>{j_1} \ar[d]_{\simeq} & \ov{\Gamma}^N_1 
\ar@{->>}[d]_{\delta} 
\ar[dr]^{\pi_2} & \\
U_{+} \ar@{^{(}->} [r]^>>>>{j} \ar@/_1.5pc/[rr] & \P^r_K \times 
\ov{\square}_K \ar@{.>}[r]^{\ \ \phi_{\eta}} & \P^r_K}
\end{equation}
such that $\delta$ is a blow-up, and in the group ${\rm Div}(\ov{\Gamma}^N_1)$ of Weil divisors, we have:
\begin{equation}\label{eqn:generic-pt-2}
\delta^*(H_{\infty} \times \ov{\square}_K) = (H_{\infty} \times \ov{\square}_K) 
+ E' \ \mbox{and}
\ \ \delta^*(\P^r_K \times \{\infty\}) = (\P^r_K \times \{\infty\}) + E'.
\end{equation}
\end{lem}

For any map $f: X \to X'$ of $K$-schemes, let $f_n$ denote the map $f \times {\rm  Id}_{\ov{\square}^n_K} :X \times \ov{\square}^n_K \to X' \times \ov{\square}^n_K$. We now show how the rational map $\phi_{\eta}:\P^r_K \times \ov{\square}_K 
\dashrightarrow \P^r_K$ eventually leads to the desired homotopy.

\begin{prop}\label{prop:moving-mod}
Let $n \geq 1$ be an integer. Let $V \inj \P^r_K \times {\square}^{n}_K$ be an 
integral closed subscheme. Assume that $V$ has modulus $H_{\infty}$ relative to 
$F^{\infty}_n$. Let $\phi_{\eta}: \A^r_K \times \square_K \to \P^r_K$ be the map 
as in \lemref{lem:generic-pt}. %Let $V'$ denote the closure of $\phi^{-1}_{\eta,n}(V)$ in $\P^r_K \times {\square}^{n+1}_K$. 
Then, the closure of $\phi^{-1} _{\eta, n} (V)$ in $\mathbb{P}^r _K \times 
\square_K ^{n+1}$ is an integral closed subscheme of 
$\P^r_K \times {\square}^{n+1}_K$ which has modulus $H_{\infty}$ relative to 
$F^{\infty}_{n+1}$.
\end{prop}

\begin{proof}We use notations of the paragraph just before Lemma 
\ref{lem:Blow-up-Final} and set 
$E'_{n} = E' \times \ov{\square}^n_K \inj \ov{\Gamma}^N_{1} 
\times  \ov{\square}^n_K$.
 
Let $\ov{V} \inj \P^r_K \times \ov{\square}^n_K$ denote the closure of $V$ and let $\nu_V: \ov{V}^N \to \P^r_K \times \ov{\square}^n_K$ denote the induced map from the normalization of $\ov{V}$. By the modulus condition, we have
\begin{equation}\label{eqn:generic-pt-2-0}
\nu^*_V(\P^r_K \times F^{\infty}_n) \ge \nu^*_V(H_{\infty} \times \ov{\square}^n_K) \ \mbox{in} \ \ {\rm Div}(\ov{V}^N). 
\end{equation}

The condition ~\eqref{eqn:generic-pt-2-0} implies that $V \cap (H_{\infty} \times \square^n_K) = \emptyset$. Set $V'= \phi^{-1}_{\eta,n}(V)$. Since $\phi_{\eta, n}$ is smooth on $\phi^{-1}_{\eta,n}(\A^r_K \times \square^n_K)$ by \lemref{lem:generic-pt}, it follows that $V'$ is an integral closed subscheme of $U_{+} \times \ov{\square}^n_K$ with $\dim_K(V') = \dim_K(V) +1$. Let $\ov{V}' \inj \P^r_K \times \ov{\square}^{n+1}_K$ be the Zariski closure of $V'$, and let $\nu_{V'}: {\ov{V}'}^N \to \ov{V}' \inj \P^r_K \times \ov{\square}^{n+1}_K$ be the induced map from the normalization of $\ov{V}'$. 
Let $W \inj \ov{\Gamma}^N_1 \times \ov{\square}^{n}_K$ be the strict transform of $\ov{V}'$. It follows from \lemref{lem:Blow-up} that $\pi_{2,n}(W \cap \delta^{-1}_n(U_{+} \times {\square}^{n}_k)) = V$. Since $\pi_{2,n}$ is projective, we must have $\pi_{2,n}(W) = \ov{V}$. This yields a commutative diagram
\begin{equation}\label{eqn:generic-pt-3}
\xymatrix@C1pc{
W^N \ar[dr]^{\nu_W} \ar[rr]^{f} \ar[dd]_{g} & & \ov{V}^N \ar[d]^{\nu_V} \\
& \ov{\Gamma}^N_{1} \times \ov{\square}^{n}_K \ar[r]^{\pi_{2,n}} \ar[d]^{\delta_{n}} & \P^r_K \times \ov{\square}^{n}_K \\
{\ov{V}'}^N \ar[r]_<<<{\nu_{V'}} & \P^r_K \times \ov{\square}^{n+1}_K, & }
\end{equation}
where $\nu_W$ is the normalization of $W$ composed with its inclusion into $\ov{\Gamma}_1 ^N \times \ov{\square}_K ^n$, and $f$ and $g$ are the maps induced by the universal property of normalization for dominant maps. Since $f$ is a surjective map of integral schemes, the condition \eqref{eqn:generic-pt-2-0} implies that $(\nu_V \circ f)^*(\P^r_K \times F^{\infty}_n) \ge (\nu_V \circ f)^*(H_{\infty} \times \ov{\square}^n_K)$ on $W^N$. In particular, we get $(\pi_{2,n} \circ \nu_W)^*(\P^r_K \times F^{\infty}_n) \ge  (\pi_{2,n} \circ \nu_W)^*(H_{\infty} \times \ov{\square}^n_K)$ on $W^N$. Equivalently,
\begin{equation}\label{eqn:generic-pt-4}
\nu^*_W(\ov{\Gamma}^N_1 \times F^{\infty}_n) \ge \nu^*_W(\pi^*_{2}(H_{\infty}) \times \ov{\square}^n_K).
\end{equation}

Since $(\phi_{\eta}|_{U_{+}})^*(H_{\infty}) = (\A^r_K \times \{\infty\}) +(H_{\infty} \times \square_K)$ by \lemref{lem:generic-pt}, we get ${j}^*_{1,n} \circ \pi^*_{2,n}(H_{\infty} \times \ov{\square}^n_K) = {j}^*_{1,n}(\P^r_K \times F^{\infty}_{n,n+1}) + {j}^*_{1,n}(H_{\infty} \times \ov{\square}^{n+1}_K)$, where $j_1: U_{+} \inj \ov{\Gamma}^N_1$ is the inclusion. Since $\P^r_K \times F^{\infty}_{n,n+1}$ and $H_{\infty} \times \ov{\square}^{n+1}_K$ are irreducible, we get $\pi^*_{2}(H_{\infty}) \times \ov{\square}^n_K \ge (\P^r_K \times F^{\infty}_{n,n+1}) + (H_{\infty} \times \ov{\square}^{n+1}_K)$ on $\ov{\Gamma}^N_{1} \times \ov{\square}^n_K$. Combining this with ~\eqref{eqn:generic-pt-4}, we get
\begin{equation}\label{eqn:generic-pt-4-0}
\nu^*_W(\ov{\Gamma}^N_{1} \times F^{\infty}_n) \ge \nu^*_W(\P^r_K \times F^{\infty}_{n,n+1}) + \nu^*_W(H_{\infty} \times \ov{\square}^{n+1}_K) \ge  \nu^*_W(H_{\infty} \times \ov{\square}^{n+1}_K). 
\end{equation}

This in turn implies that 
\[
\begin{array}{lll}
(\delta_{n} \circ \nu_W)^*(\P^r_K \times F^{\infty}_{n+1}) & = &
(\delta_{n} \circ \nu_W)^*(\P^r_K \times F^{\infty}_n \times \ov{\square}_K) \\
& &
+ (\delta_{n} \circ \nu_W)^*(\P^r_K \times \ov{\square}^n_K \times \{\infty\}) \\
& = & \nu^*_W(\ov{\Gamma}^N_{1} \times F^{\infty}_n)
+ (\delta_{n} \circ \nu_W)^*(\P^r_K \times \ov{\square}^n_K \times \{\infty\}) \\
& \ge & 
\nu^*_W(H_{\infty} \times \ov{\square}^{n+1}_K) + 
(\delta_{n} \circ \nu_W)^*(\P^r_K \times \ov{\square}^n_K \times \{\infty\}) \\
& {=}^{\dagger} & 
\nu^*_W(H_{\infty} \times \ov{\square}^{n+1}_K) + \nu^*_W(E'_{n}) +
\nu^*_W(\P^r_K \times \ov{\square}^n_K \times \{\infty\}) \\
& {=}^{\ddagger} & 
(\delta_{n} \circ \nu_W)^*(H_{\infty} \times \ov{\square}^{n+1}_K)
+ \nu^*_W(\P^r_K \times \ov{\square}^n_K \times \{\infty\})  \\
& \ge & (\delta_{n} \circ \nu_W)^*(H_{\infty} \times \ov{\square}^{n+1}_K), \\
\end{array}
\]
where ${=}^{\dagger}$ and ${=}^{\ddagger}$ follow from \lemref{lem:Blow-up-Final}. Using ~\eqref{eqn:generic-pt-3}, this gives $g^*(\nu^*_{V'}(\P^r_K \times F^{\infty}_{n+1})) \ge g^*(\nu^*_{V'}(H_{\infty} \times \ov{\square}^{n+1}_K))$. Since $g$ is surjective map of integral normal schemes, we 
conclude by \lemref{lem:cancel} that 
$\nu^*_{V'}(\P^r_K \times F^{\infty}_{n+1}) \ge  \nu^*_{V'}(H_{\infty} \times 
\ov{\square}^{n+1}_K)$.
\end{proof}

\begin{thm}\label{thm:moving-PS}
Given an integer $r \ge 1$, let $D \inj \P^r_k$ be a hyperplane. Let $\sW = \{W_1, \cdots , W_s\}$ be a finite collection of locally closed subsets of $\P^r_k$ and let $e: \sW \to \Z_{\ge 0}$ be a set function. Then, the inclusion $z^q_{\sW}(\P^r_k|D, \bullet) \inj z^q_{\sW,e}(\P^r_k|D, \bullet)$ is a quasi-isomorphism. In particular, the inclusion $z^q_{\sW}(\P^r_k|D, \bullet) \inj z^q(\P^r_k|D, \bullet)$ is a quasi-isomorphism.
\end{thm}

\begin{proof}
The second part follows easily from the first part because $z^q(\P^r_k|D, \bullet) = z^q_{q} (X|D, \bullet)$. We shall prove the first part of the theorem in several steps. We can find a linear automorphism $\tau: \P^r_k \xrightarrow{\simeq} \P^r_k$ such that $\tau(D) = H_{\infty}$. Replacing $\sW$ by $ \tau (\sW)$, we reduce to the case when $D= H_{\infty}$, which we suppose from now. In view of \propref{prop:spread}, we only need to show that the map ${\rm pr}^*_{K/k}:\frac{z^q_{\sW,e}(\P^r_k|D, \bullet)}{z^q_{\sW}(\P^r_k|D, \bullet)}\to \frac{z^q_{\sW_K,e}(\P^r_K|D_K, \bullet)}{z^q_{\sW_K}(\P^r_K|D_K, \bullet)}$ is zero on the homology, where we choose $K = k(\P^r_k)$.

Following the notations so far in this section, consider the maps
\[
\A^r_K \times \square^{n+1}_K \xrightarrow{\phi_{\eta,n}}
\P^r_K \times \square^n_K \xrightarrow{{\rm pr}_{K/k}}
\P^r_k \times \square^n_k.
\]
For any irreducible cycle $V \inj \P^r_k \times \square^n_k$, let $H^*_n(V)=({\rm pr}_{K/k} \circ \phi_{\eta,n})^{-1}(V)$ and let $\ov{H}^*_n(V)$ be its closure in $\P^r_K \times \square^{n+1}_K$. We can extend this linearly to cycles in $z^q(\P^r_k|D, n)$.

Suppose $V$ is an irreducible cycle in $z^q_{\sW,e}(\P^r_k|D, n)$. We claim:
\begin{enumerate}
\item $\ov{H}^*_n(V) \in z^q_{\sW_K,e}(\P^r_K|D_K, n+1)$.
\item $\ov{H}^*_n(V) \in z^q_{\sW_K}(\P^r_K|D_K, n+1)$ if $V \in z^q_{\sW}(\P^r_k|D, n)$.
\item $\iota^*_{n+1, n+1, 0}(\ov{H}^*_n(V)) = V$ and $\iota^*_{n+1, n+1, 1}(\ov{H}^*_n(V)) \in z^q_{\sW_K}(\P^r_K|D_K, n)$.
\end{enumerate}

%First observe that the modulus condition for $\ov{H}^*_n(V)$ follows from \propref{prop:moving-mod}. 
We now prove this claim using the previous results of this section.
Since $V$ has modulus $D$ on $\P^r_k \times \square^{n}_k$, it follows that $V$ is a closed subscheme of $\A^r_k \times \square^n_k$. In particular, $V \in z^q_{\sW^0,e}(\A^r_k, n)$, where $\sW^0 = \{W_1 \cap \A^r_k, \cdots , W_s \cap \A^r_k\}$. Since $\ov{H}^*_n(V)$ has modulus $D$ on $\P^r_K \times \square^{n+1}_K$ by \propref{prop:moving-mod}, it follows that $\ov{H}^*_n(V)$ is an integral closed subscheme of $\A^r_K \times \square^{n+1}_K$. In particular, $\ov{H}^*_n(V) = H^*_n(V)$. This shows that we can replace $\P^r_k$, $\ov{H}^*_n(V)$ and $\sW$ by $\A^r_k$, $H^*_n(V)$ and $\sW^0$, respectively, to prove the claim.

We prove (3) first. By the definition of $\phi_{\eta}$, we have $\iota^*_{n+1, n+1, 0}({H}^*_n(V)) = V$. In particular, $H^*_n(V)$ intersects $F_{n+1, n+1,0}$ and its all faces properly. We thus have to show that $\iota^*_{n+1, n+1, 1}({H}^*_n(V)) \in z^q_{\sW^0_K}(\A^r_K|D_K, n)$ to prove (3). 

Let $\A^r_k$ act on itself by translation and let it act on $\A^r_k \times \square^{n}_k$ by acting trivially on $\square^{n}_k = \square^{n}_k \times \{1\} \inj \square^{n+1}_k$. Consider the map $\psi: \A^r_K \times \square^{n}_K \to \A^r_K$ defined by $\psi(x,y) = \eta$. One checks that the assumptions of \corref{cor:Bloch-2} are satisfied. Applying this corollary to each $A = \ov{W_i \cap \A^r_k}$ (where the closure is taken in $\A^r_k$) and $B = \A^r_k \times F$ for any face $F$ of $\square^{n}_k \times \{1\}$, we deduce $\iota^*_{n+1, n+1, 1}({H}^*_n(V)) \in z^q_{\sW^0_K}(\A^r_K|D_K, n)$. We have thus proven (3). Since (2) is a special case of (1) where we take $e= 0$, we are left with proving (1).

To prove (1), it is enough to consider the case when $\sW = \{W\}$ is a singleton. Note $V \in z^q_{W,e}(\A^r_k, n)$ and let $F \inj \square^{n+1}_K$ be any face. If $F \inj \square^n_K \times \{0\}$, then the intersection ${H}^*_n(V) \cap (W \times F)$ has the desired dimension because $\iota^*_{n+1, n+1, 0}({H}^*_n(V)) = V$ and $V \in z^q_{W,e}(\A^r_k, n)$. We have already proven in (3) that the intersection ${H}^*_n(V) \cap (W \times F)$ is proper if $F \inj \square^n_K \times \{1\}$. We can thus assume that $F = F'_K \times \square_K$, where $F'$ is a face of $\square^n_k$.

Set $Z = V \cap (\A^r_k \times F')$. Consider the map $\psi: \A^r_K \times \square_K \times F'_K \to \A^r_K$ defined by $\psi(x,t, y) = \eta t$ and let $\theta: \A^r_K \times \square_K \times F'_K \to \A^r_K \times \square_K \times F'_K$ be given by $\theta(x,t, y) = (x + \eta t, t, y)$. Let $\A^r_k$ act by translation on itself and trivially on $\square_k \times  F'$. Then $\theta(x,t, y) = \psi(x,t, y) \cdot (x,t, y)$. Applying \lemref{lem:Bloch-1} with $X = \A^r_k \times \square_k \times F', \ A = \ov{W} \times \square_k \times F', U = \A^r_k \times \G_{m,k} \times F',$ and $B = (V \times \square_k) \cap F_k = Z \times \square_k \inj X \times F',$
%\[
%X = \A^r_k \times \square_k \times F', \ A = \ov{W} \times \square_k \times F', \ \
%U = \A^r_k \times \G_{m,k} \times F',
%\]
%\[
%\mbox{and} \ \ 
%B = (V \times \square_k) \cap F_k = Z \times \square_k \inj X \times F',
%\]
it follows that the intersection $\theta(A_K) \cap B_K$ is proper away from $\A^r_K \times \{0\} \times F'_K$, i.e., the intersection $(H^*_n(V) \cap F) \cap (W_K \times F)$ is proper away from $\A^r_K \times \{0\} \times F'_K$.

On the other hand, as $V \in z^q_{W,e}(\A^r_k, n)$ and so $V$ meets $W \times F'$ in excess dimension at most $e(W)$, it follows that $H^*_n(V) \cap F$ must meet $W \times F$ in excess dimension at most $e(W)$ along $\A^r_K \times \{0\} \times F'_K$. Thus $H^*_n(V)$ intersects $W_K \times F_K$ in excess dimension at most $e(W)$ for all faces $F_K \inj \square^{n+1}_K$. In other words, $H^*_n(V) \in z^q_{W_K,e}(\A^r_K, n+1)$. This proves (1) and hence the claim.

It follows from the claim that there is a chain homotopy
\[
H^*_{\eta}: \frac{z^q_{\sW,e}(\P^r_k|D, \bullet)}{z^q_{\sW}(\P^r_k|D, \bullet)}\to \frac{z^q_{\sW_K,e}(\P^r_K|D_K, \bullet)}{z^q_{\sW_K}(\P^r_K|D_K, \bullet)}[-1]
\]
and composed with the restriction map $\{1\} \inj \square_k$, there is a chain map
\[
H^*_{\eta, 1}: \frac{z^q_{\sW,e}(\P^r_k|D, \bullet)}{z^q_{\sW}(\P^r_k|D, \bullet)} \to \frac{z^q_{\sW_K,e}(\P^r_K|D_K, \bullet)}{z^q_{\sW_K}(\P^r_K|D_K, \bullet)}
\]
such that $H^*_{\eta} \circ \partial + \partial \circ H^*_{\eta} ={\rm pr}^*_{K/k} - H^*_{\eta, 1}$. Since $H^*_{\eta, 1} = 0$ by the claim, we see that ${\rm pr}^*_{K/k}$ is zero on the homology. The proof of the theorem is complete.
\end{proof}

\begin{cor}\label{cor:moving-PS-Gen}
Given an integer $r \ge 1$, let $D \inj \P^r_k$ be a hyperplane. Let $\sW = \{W_1, \cdots , W_s\}$ be a finite collection of locally closed subsets of $\P^r_k$ and let $e: \sW \to \Z_{\ge 0}$ be a set function. Then, the inclusion $z^q_{\sW,e-1}(\P^r_k|D, \bullet) \inj z^q_{\sW,e}(\P^r_k|D, \bullet)$ is a quasi-isomorphism.
\end{cor}

\begin{proof}
For every $e: \sW \to \Z_{\ge 0}$, there is a short exact sequence of chain complexes
\begin{equation}\label{eqn:moving-PS-Gen-0}
0 \to \frac{z^q_{\sW,e-1}(\P^r_k|D, \bullet)}{z^q_{\sW}(\P^r_k|D, \bullet)}\to \frac{z^q_{\sW,e}(\P^r_k|D, \bullet)}{z^q_{\sW}(\P^r|D, \bullet)} \to \frac{z^q_{\sW,e}(\P^r_k|D, \bullet)}{z^q_{\sW,e-1}(\P^r_k|D, \bullet)}\to 0.
\end{equation}
The first two quotient complexes are acyclic by \thmref{thm:moving-PS}. Hence the last one must be acyclic as well.
\end{proof}

\section{Moving lemma for projective schemes}\label{sec:ML-proj}
In this section, we prove the moving lemma for the higher Chow groups of 
projective schemes with very ample modulus. We assume for a while that the base field $k$ is infinite. This is only a temporary assumption and
will be removed in the final statement of the moving lemma
(see \thmref{thm:Main-moving}). 

%\subsection{Notations}
We fix a closed embedding $\iota_X: X \inj \P^N_k$ of an equidimensional 
reduced projective 
scheme $X$ over $k$ of dimension $d \ge 1$, with $d < N$. 
We fix two distinct hyperplanes $H_m, H_{\infty} \inj \P^N_k$ and let 
$L_{m,\infty} = H_m \cap H_{\infty} \in \Gr(N-2, \P^N_k)$. 
We may assume that $X \not \subset H_m \cup H_{\infty}$. 
We set 
\[
X_0 = X \setminus H_{\infty}\stackrel{j_0}{\inj} X, 
U = X \setminus H_m, U_0 = U \cap X_0, D = \iota^*_X(H_m) \ \mbox{and} \
D_0 = j^*_0(D)
\] 
so that $X = U \cup D $ and $X_0 = U_0 \cup D_0.$ 
We shall assume that $U$ is smooth over $k$.
(N.B. The hyperplane $H_m$ could have been just called $H$, but we insisted on putting the subscript $m$ to remind ourselves psychologically that this $H_m$ later induces the modulus divisor.)

%Set 
%\[
%X_0 = X \setminus H_{\infty}\stackrel{j_0}{\inj} X, \ \ 
%U = X \setminus H_m, \ \ U_0 = U \cap X_0, \ \
%D = \iota^*_X(H_m) \ \mbox{and} \ \ D_0 = j^*_0(D)
%\]
%\[
%\mbox{so \ that} \ \ X = U \cup D \ \mbox{and} \ \ 
%X_0 = U_0 \cup D_0. 
%\]
 
Given a locally closed subset $S \subsetneq \P^N_k$, let $\Gr(S, n, \P^N_k)$ denote the set of $n$-dimensional linear subspaces of $\P^N_k$ which do not intersect $S$. Denote the set of $n$-dimensional linear subspaces of $\P^N_k$ containing a locally closed subscheme $S \subsetneq \P^N_k$ by $\Gr_S(n, \P^N_k)$. We let $\dim(\emptyset) = -1$. Given two locally closed subsets $Z_1, Z_2 \inj \P^N_k$, let $\Sec(Z_1, Z_2)$ denote the union of all lines $\ell_{xy} \inj \P^N_k$, joining $x \in Z_1$ and $y \in Z_2$ with $x \neq y$. One checks that $\dim(\Sec(Z_1,Z_2)) = \dim(Z_1) + \dim(Z_2) - \dim(Z_1 \cap Z_2)$ if $Z_1$ and $Z_2$ are linear subspaces of $\P^N_k$. In general, we have $\dim(\Sec(Z_1,Z_2)) \le \dim(Z_1) + \dim(Z_2) + 1$. Given a closed point $x \in X$, let $T_x(X)$ denote the union of lines in $\P^N_k$ which are tangent to $X$ at $x$. For any locally closed subset $Y \subseteq X$, let $T_Y(X) = \bigcup_{x \in Y} T_x(X)$. It is clear that $\dim(T_Y(X)) \le \dim(Y) + d$ if $Y \subseteq U$. 
With these notations, we first prove the following. 

%\enlargethispage{25pt}

\begin{lem}\label{lem:fs-open-1}
Let $W \inj \P^N_k$ be a closed subscheme of dimension at most $d$ such that $W \not \subset H_m$. Then, $\Gr(W, N-d-1, H_m)$ is a dense open subset of $\Gr(N-d-1, H_m)$. If $L_{m,\infty}$ intersects $W$ properly, then $\Gr(W, N-d-1, L_{m, \infty})$ is a dense open subset of $\Gr(N-d-1, L_{m, \infty})$. 
\end{lem}

\begin{proof}
Consider the incidence variety $S = \{(x, L) \in W \times \Gr(N-d-1, H_m)| x \in L\}$. We have the projection maps of projective schemes
\begin{equation}\label{eqn:fs-open-1-0} 
\xymatrix@C1pc{
W & S \ar[r]^<<<{\pi_2} \ar[l]_{\pi_1} &  \Gr(N-d-1, H_m).}
\end{equation}

The fiber of $\pi_1$ over $W \setminus H_m$ is empty and it is a smooth fibration over $(W \cap H_m)_{\rm red}$ with each fiber isomorphic to $\Gr(N-d-2, \P^{N-2}_k)$. It follows that $\dim(S) = \dim(W \cap H_m) + d(N-d-1) \le d + d(N-d-1) -1 = d(N-d) -1$. Thus $\pi_2(S)$ is a closed subscheme of $\Gr(N-d-1, H_m)$ of dimension at most $d(N-d) -1$. On the other hand, $\dim(\Gr(N-d-1, H_m)) = d(N-d)$ so that $\Gr(W, N-d-1, H_m)$ is dense open in $ \Gr(N-d-1, H_m) \setminus \pi_2(S)$.

If $L_{m,\infty}$ intersects $W$ properly, then we can argue as above with $H_m$ replaced by $L_{m, \infty}$. We find in this case that $\dim(\pi_2(S)) \le \dim(S) = \dim(W \cap L_{m, \infty}) + (d-1)(N-d-1) \le d + (d-1)(N-d-1) -2 = (d-1)(N-d) -1$. Since $\dim(\Gr(N-d-1, L_{m, \infty})) = (d-1)(N-d)$, we arrive at the desired conclusion.
\end{proof}

Given an inclusion of linear subspaces $L \subsetneq L' \subseteq \P^N_k$ such that $\dim(L) \le N-d-1$ and $X \cap L = \emptyset$, the linear projection away from $L$ defines a Cartesian diagram
\begin{equation}\label{eqn:proj-0}
\xymatrix@C1pc{
X\setminus L' \ar[r] \ar[d] & X \ar[d]^{\phi_L} & X \cap L' \ar[l] \ar[d] \\
\P^d_k \setminus L' \ar[r] & \P^d_k & \P^d_k \cap L' \ar[l]}
\end{equation}
of finite maps, where $\P^d_k \inj \P^N_k$ is a linear subspace complementary to $L$. Let $R_L(X) \subset X$ denote the ramification locus of $\phi_L$.

For an irreducible locally closed subset $A \subsetneq X$, let $L^{+}(A)$ denote the closure of $\phi^{-1}_L(\phi_L(A)) \setminus A$ in $\phi^{-1}_L(\phi_L(A))$. We linearly extend this definition to all cycles on $X$. We shall use similar notation for locally closed subsets of $X \times \square^n$ with $\phi_L$ replaced by $\phi_L \times {\rm Id}_{\square^n}$.

For two  locally closed subsets $A, C \subset X$, let $e(A,C) = {\rm max}\{\dim(Z)- \dim(A)-\dim(C) + d\}$, where the maximum is taken over all irreducible components $Z$ of $A \cap C$, assuming these numbers are non-negative. Else, we take $e(A,C)$ to be zero.

\begin{lem}\label{lem:Levine}
Let $A \subsetneq X \setminus H_m$ be an irreducible locally closed subset and let $C \subsetneq X \setminus H_m$ be any locally closed subset. Let $\Sigma = \{x_1, \cdots , x_s\}$ be a set of distinct closed points of $X$ contained in $A$. Then, there is a dense open subset $\sU^{A,C}_X \inj \Gr(N-d-1, H_m)$ such that the following hold for every $L \in \sU^{A,C}_X$. 
\begin{enumerate}
\item $X \cap L = \emptyset$.
\item $R_L(X)$ contains no irreducible component of $A, C$ or $A \cap C$.
\item
$R_L(X) \cap \Sigma = \emptyset$.
\item $e(L^{+}(A) \cap C) \le {\rm max}\{e(A,C)-1,0\}$.
\item The map $k(\phi_L(x)) \to k(x)$ is an isomorphism for $x \in \Sigma$. 
\end{enumerate}
\end{lem}

\begin{proof}
The item (1) follows from \lemref{lem:fs-open-1}. So we prove 
the remaining ones. We may assume that $C$ is irreducible. Let $L \in \Gr(X, N-d-1, H_m)$. Set $T^L_r = R_L(X) \cap A \cap C = R_L(U) \cap A \cap C$ and $T^L_{ur} = (L^{+}(A) \cap C) \setminus T^L_r$. (N.B. `$r$' is for ramified and `$ur$' is for unramified.) Then we must have $L^{+}(A) \cap C \subseteq  T^L_{ur} \cup T^L_r$ and hence $\dim(L^{+}(A) \cap C) \le {\rm max}\{\dim(T^L_{ur}), \dim(T^L_r)\}$. 
Since the left square in ~\eqref{eqn:proj-0} is Cartesian (where $L' = H_m$)
and $A, C \subset U = X \setminus H_m$, 
it follows that the loci $T^L_r$ and $T^L_{ur}$ are contained in $U$. 

Let $S \inj ((A \times C) \setminus \Delta_X) \times \Gr(N-d-1, H_m)$ be the incidence variety $S = \{(a,c, L)| \ell_{ac} \cap L \neq \emptyset\}$. 
%and let $\ov{S}$ the closure of $S$ in $A \times C \times \Gr(N-d-1, H_m)$. 
We have the projections $A \times C \xleftarrow{pr_1} S \xrightarrow{pr_2} \Gr(N-d-1, H_m)$.

Since $L \cap X = \emptyset$, we see that for any point $(a,c) \in ((A \times C) \setminus \Delta_X)$, $pr^{-1}_1((a,c)) = \{L \in \Gr(N-d-1, H_m)| \dim(L \cap \ell_{ac}) = 0\}$. Consider the map $\pi: pr^{-1}_1((a,c))   \to \ell_{ac}$ given by $\pi(L) = L \cap \ell_{ac}$.

Our hypothesis says that $(A \cup C) \cap H_m = \emptyset$ and this implies that $\ell_{ac} \not \subset H_m$. In particular, $x_{ac} = \ell_{ac} \cap H_m$ is a single closed point of $\P^N_k$. This implies that $\pi^{-1}(\ell_{ac} \setminus \{x_{ac}\}) = \emptyset$ and $\pi^{-1}(\{x_{ac}\}) = pr^{-1}_1((a,c)) = \{L \in  \Gr(N-d-1, H_m)| x_{ac} \in L\} \simeq \Gr(N-d-2, \P^{N-2}_k)$. It follows that $\dim(pr^{-1}_1((a,c))) = (N-d-1)(N-2-(N-d-2))= d(N-d-1)$. We conclude from this that
\begin{equation}\label{eqn:Levine-0}
\begin{array}{lll}
\dim(S) & \le & \dim(A) + \dim(C) + d(N-d-1) \\
& = & \dim(A) + \dim(C) + d(N-d) - d \\
& = &  \dim(A) + \dim(C) + \dim(\Gr(N-d-1, H_m)) - d.
\end{array}
\end{equation}

Let $p_C: S \to A \times C \to C$ be the composite projection. We now observe that $c \in T^L_{ur}$ if and only if there exists $a \in A$ such that $a \neq c$ and $\ell_{ac} \cap L \neq \emptyset$. Since $c \in C$ as well, this means that $(a,c) \in pr^{-1}_2(L)$. In other words, $T^L_{ur} \subset p_C(pr^{-1}_2(L))$. On the other hand, it follows from \eqref{eqn:Levine-0} that there is a dense open subset $\sU^{A,C}_{ur} \subseteq \Gr(N-d-1, H_m)$ such that $pr^{-1}_2(L)$ is either empty or has dimension $\dim(A) + \dim(C) - d$ for every $L \in \sU^{A,C}_{ur}$. We conclude: 

$(\star)$ There is a dense open subset $\sU^{A,C}_{ur} \subseteq 
\Gr(N-d-1, H_m)$ such that $\dim(T^L_{ur}) \le \dim(A) + \dim(C) - d$ for each 
$L \in \sU^{A,C}_{ur}$.

Since $U$ is smooth, given any point $x \in A \cap C$, our hypothesis implies that $T_x(X)$ is a locally closed subscheme of $\P^N_k$ of dimension $d$ such that $T_x(X) \not \subset H_m$. We can therefore apply \lemref{lem:fs-open-1} to find a dense open subset of $\Gr(N-d-1, H_m)$ whose elements do not meet $T_x(X)$. But this means that $x \notin R_L(X)$ for every $L$ in this dense open subset. We can repeat this for any chosen point in $A$ and $C$ as well. 
Since $\Sigma \subset A$, we therefore conclude:

$(\star \star)$ There is a dense open subset $\sU^{A,C}_{r} \subseteq \Gr(N-d-1, H_m)$ such that $R_L(X)$ does not contain any component of 
$A, C$ or $A \cap C$ and it does not intersect $\Sigma$,
whenever $L \in \sU^{A,C}_{r}$. 

For any $L \in \sU^{A,C}_{r}$, we have $\dim(T^L_r) = \dim(R_L(X) \cap A \cap C) \le {\rm max}\{\dim(A \cap C) - 1, 0\}$. Combining $(\star)$ and $(\star \star)$ with \lemref{lem:fs-open-1} and setting $\sU^{A,C}_X = \sU^{A,C}_{ur} \cap  \sU^{A,C}_{r}$, we conclude that $\sU^{A,C}_X$ is a dense open subset of $\Gr(N-d-1, H_m)$ such that $e(L^{+}(A) \cap C) \le {\rm max}\{e(A,C)-1,0\}$ for $L \in \sU^{A,C}_X$.

The proof of (5) is clear if $k$ is algebraically closed. In general, let $\ov{k}$ be an algebraic closure of $k$ and let $\pi_Y: Y_{\ov{k}} \to Y$ denote the base change to $\ov{k}$ for any $Y \in \Sch_k$. 
%In the following, we use an argument similar to Step 2 of the proof of \cite[Lemma 7.3.2]{KPcrys}. 
For any $x \in \Sigma$, let $S_x = \pi^{-1}_X(x)$ and let $S =\bigcup_{x \in \Sigma} \ S_x$. Then $S \inj X_{\ov{k}}$ is a finite set of closed points contained in $A_{\ov{k}}$. Let $W'$ be the union of lines $l_{xy}$ in $\P^N_{\ov{k}}$ such that $x \neq y \in S$. Since $S \subset A_{\ov{k}}$ and $A \cap H_m = \emptyset$, we see that $W' \not \subset H_{m, \ov{k}}$. Since $d \ge 1 = \dim(W')$, we can apply \lemref{lem:fs-open-1} to assume that $W' \cap L = \emptyset$ for all $L \in \sU^{A,C}_{X_{\ov{k}}}:= \sU^{A_{\ov{k}},C_{\ov{k}}}_{X_{\ov{k}}}$.

Since $\Gr(N-d-1, H_{m, \ov{k}})$ contains an affine space $\A^{d(N-d)}_{\ov{k}}$ as a dense open subset, we can replace $\sU^{A,C}_{X_{\ov{k}}}$ by $\sU^{A,C}_{X_{\ov{k}}} \cap \A^{d(N-d)}_{\ov{k}}$ and assume that $\sU^{A,C}_{X_{\ov{k}}} \subseteq \A^{d(N-d)}_{\ov{k}}$. Since $k$ is infinite, the set of points in $\A^{d(N-d)}_{\ov{k}}$ with coordinates in $k$ is dense in $\A^{d(N-d)}_{\ov{k}}$. Hence, there is a dense subset of $\sU^{A,C}_{X_{\ov{k}}}$ each of whose points $L$ is defined over $k$, i.e., $L \in \Gr(N-d-1, H_{m})$. Let $L \in \Gr(N-d-1, H_{m})$ be such that (1) $\sim$ (4) hold and $W' \cap L_{\ov{k}} = \emptyset$. We consider the Cartesian square
\begin{equation}\label{eqn:base-change}
\xymatrix@C1pc{
X_{\ov{k}} \ar[d]_{\pi_X} \ar[r]^{\phi_{L_{\ov{k}}}} & \P^d_{\ov{k}} 
\ar[d]^{\pi_{\P^d}} \\
X \ar[r]^{\phi_L} & \P^d_k.}
\end{equation}

{\bf Claim:} \emph{For a closed point $x \in U$ and $y: = \phi_L(x)$, one has $|\pi_{\P^d}^{-1}(y)| \le |\pi_X^{-1}(x)|$, and the equality holds if and only if 
${[k(x): k(y)]}^{\rm sep} = 1$. Furthermore, this equality holds if the map 
$\phi_{L_{\ov{k}}}: \pi_X^{-1}(x) \to \pi_{\P^d}^{-1}(y)$ is injective.}

It is an elementary fact that $|\pi_X^{-1}(x)| = {[k(x):k]}^{\rm sep}$ and 
$|\pi_{\P^d}^{-1}(y)| = {[k(y) : k]}^{\rm sep}$. The inclusions $k \inj k(y) \inj k(x)$ and therefore the equality ${[k(x):k]}^{\rm sep} = {[k(y) : k]}^{\rm sep}
\cdot {[k(x) : k(y)]}^{\rm sep}$ implies the first assertion. 
Next, the injectivity of the map 
$\phi_{L_{\ov{k}}}: \pi_X^{-1}(x) \to \pi_{\P^d}^{-1}(y)$ implies that 
$|\pi_{\P^d}^{-1}(y)| \ge |\pi_X^{-1}(x)|$. The second part of the Claim follows.

To prove (5) in general, it suffices to show that the finite field extension  
$k(\phi_L(x)) \inj k(x)$ is separable as well as purely inseparable for
each $x \in \Sigma$. Now, the separability of this extension  
is equivalent to the assertion $x \notin R_L(X)$, and this is 
guaranteed by (3). To prove inseparability, it is enough to show,
using the above claim, that 
$\phi_{L_{\ov{k}}}: \pi_X^{-1}(x) \to \pi_{\P^d}^{-1}(\phi_L(x))$ is injective. 
But this follows immediately from the fact that 
$W' \cap L_{\ov{k}} = \emptyset$. The proof of the lemma is complete.
\end{proof}

\begin{lem}\label{lem:Levine-birational-missing}
Let $\alpha \in z^q(X|H_m, n)$ be an admissible cycle. Let $C \subset X \setminus H_m$ be a locally closed subset as in \lemref{lem:Levine}. We can then find a dense open subset $\sU^{Z,C}_X \subset \Gr(N-d-1, H_m)$ such that the following hold for every $L \in \sU^{Z,C}_X$.
\begin{enumerate}
\item $X \cap L = \emptyset$.
\item For every irreducible component $Z$ of $\alpha$, no irreducible component of the support of the cycle $\phi^*_L \circ \phi_{L *}([Z]) -[Z]$ coincides with $Z$.
\end{enumerate}
\end{lem}

\begin{proof}
It is enough to consider the case when $\alpha = [Z]$ is an irreducible admissible cycle. For any $L \in \Gr(N-d-1, H_m)$ satisfying (1), we need to prove the following to achieve (2):
\begin{enumerate}
\item [(i)] The ramification locus $R^n_L(X)$ of $\phi^n_L$ does not contain $Z$, where $\phi^n_L:=\phi_L\times {\rm Id}_{\square^n_k}$.

\item [(ii)] $\phi^n_L|_Z: Z \to \phi^n_L(Z)$ is birational.
\end{enumerate}

Let ${\rm pr}_X: X \times \square^n_k \to X$ and ${\rm pr}_{\square^n_k}: X \times \square^n_k \to \square^n_k$ be the projection maps. We fix a closed point $z \in Z$ and set $x = {\rm pr}_X(z), \ y = {\rm pr}_{\square^n_k}(z), W = \phi^n_L(Z)$ and $A = {\rm pr}_X(Z)$. Then $A$ is a finite disjoint union of locally closed subsets of $X$. Since $Z$ is an admissible cycle having modulus $H_m$, 
we must have $A \cap H_m = \emptyset$. In particular, $x \in U$.
It is shown in the proof of \thmref{thm:small} that $(\{y\} \times X) \cap Z$ is a finite set of closed points away from $(\{y\} \times H_m)$. In particular, $D:= {\rm pr}_X((\{y\} \times X) \cap Z)$ is a finite set of closed points of $X$ containing $x$ and contained in $A$. This implies that $\Sec(x, D)$ is a closed subset of $\P^N_k$ of dimension one which is not contained in $H_m$. 
Hence, we conclude from \lemref{lem:fs-open-1} that $\Gr(\Sec(x,D), N-d-1,H_m)$ is dense open in $\Gr(N-d-1, H_m)$.

We have shown in the proof of \lemref{lem:Levine} that there is a dense open subset $\sU_{Z,1} \subset \Gr(N-d-1, H_m)$ such that $T_x(X) \cap L = \emptyset$ for each $L \in \sU_{Z,1}$. Since the left square in ~\eqref{eqn:proj-0} is
Cartesian and $\phi_L$ is finite, it follows that its restriction
$\phi^U_L: U \to \P^d_k \setminus H_m$ is also finite. Since $U$ is furthermore
smooth, it follows that $\phi^U_L$ is a finite and flat morphism of smooth
schemes.

The flatness of $\phi^U_L$ now implies that there is an open neighborhood 
$V \subset U$ of $x$ such that $\phi_L:V \to \P^d_k$ is {\'e}tale. In particular, $\phi^n_L:V \times \square^n_k \to \P^d_k \times \square^n_k$ is {\'e}tale. This implies that there is an open subset $V'$ of $Z$ containing $z$ such that 
$\phi^n_L|_{V'}: V' \to W$ is unramified. We set 
$\sU^{Z,C}_X = \Gr(\Sec(x,D), N-d-1,H_m) \cap \sU_{Z,1} \cap \sU^{A,C}_X$, 
where $\sU^{A,C}_X$ is as in \lemref{lem:Levine}.

We fix any $L \in \sU^{A,C}_X$. Since $R^n_L(X) = R_L(X) \times \square^n_k$ and no component of $A$ is in $R_L(X)$ by \lemref{lem:Levine}, it follows that $Z \not \subset R^n_L(X)$, proving (i). To prove (ii), it suffices to show that $z \notin R^n_L(Z)$, $\{z\} = (\phi^n_L)^{-1}(\phi^n_L(z)) \cap Z$ and $k(\phi^n_L(z)) \xrightarrow{\simeq} k(z)$, because they imply that the map $\sO_{W, \phi^n_L(z)} \to \sO_{Z,z}$ is an isomorphism, and hence induces isomorphism of the function fields.

We have shown above that $z \notin R^n_L(Z)$. Since the map $k(\phi_L(x)) \to k(x)$ is an isomorphism by \lemref{lem:Levine}, it follows that the map $\phi^n_L:\square^n_{k(x)} \to \square^n_{k(\phi_L(x))}$ is also an isomorphism. In particular, the map $k(\phi^n_L(z)) \to k(z)$ is an isomorphism. To show $\{z\} = (\phi^n_L)^{-1}(\phi^n_L(z)) \cap Z$, note that if there is a closed point $z' \in ((\phi^n_L)^{-1}(\phi^n_L(z)) \cap Z) \setminus \{z\}$, then $x':={\rm pr}_X(z') \in D \cap L^{+}(x)$, where recall that $L^{+}(x) = \phi_L^{-1}(\phi_L(x)) \setminus \{x\}$. But this can happen only if $\ell_{xx'} \cap L \neq \emptyset$, which is not the case because $L \in \Gr(\Sec(x,D), N-d-1,H_m)$. This finishes the proof of (ii) and the lemma.
\end{proof}

\begin{remk}\label{remk:Levine-*}
We make few comments on \lemref{lem:Levine}. 
To some readers, this result may appear similar to \cite[Lemma~3.5.4]{Levine}. But we caution the reader that the context, the underlying hypotheses and the proofs of the two results are different. We explain these differences. 
\begin{enumerate}
\item The proof of \lemref{lem:Levine} does not work if we replace $X$ by $X \cap \A^N_k$. The reason is that even if $X$ intersects $L_{m, \infty}$ properly, we may not be able to find points on $A\cap C$ whose tangent spaces will intersect $L_{m, \infty}$ properly and this will force the second part of the proof of \lemref{lem:Levine} to break down.

Since \emph{loc.cit.} considers the affine case, Levine can not therefore use the above argument. Instead, he uses the idea of reimbedding $X$ into a big enough projective space which allows him to take care of the above intersection problem associated to the tangent spaces.

\item Contrary to \emph{loc.cit.}, we can not use the reimbedding idea. The reason is that we may not be able to realize our modulus $H_m$ as pull-back of any hypersurface on the bigger projective space under the reimbedding. This in turn may not allow us to realize $H_m$ as pull-back of a hypersurface under a linear projection.  

\item The modulus condition imposes more severe restrictions on the choice of $L$ than in the situation of \emph{loc.cit.} So we need to make more refined choices and without changing the given embedding of $X$.
\end{enumerate}
\end{remk}

Let $\sW = \{W_1, \cdots , W_s\}$ be a finite collection of locally closed subsets of $X \setminus H_m$ and let $e:\sW \to \Z_{\ge 0}$ be a set function.

Let $K$ denote the function field of $\Gr(N-d-1, H_m)$ and let $L_{\rm gen} \in \Gr(N-d-1, H_m)(K)$ be the generic point of $\Gr(N-d-1, H_m)$. This can be seen as a $K$-rational point of $\Gr(N-d-1, H_m)$.

\begin{lem}\label{lem:Levine-box}
The linear projection away from $L_{\rm gen}$ defines a finite map 
$\phi_{L_{\rm gen}}: X_K \to \P^d_K$ satisfying the following conditions.
\begin{enumerate}
\item
The restriction $\phi^U_{L_{\rm gen}}: U_K \to \P^d_K \setminus H_{m,K}$ is 
finite and flat.
\item $D_K = \phi^*_{L_{\rm gen}}(H_{\rm gen})$ for the hyperplane 
$H_{\rm gen} = (H_m \cap \P^d)_K$ in $\P^d_K$.
\item The pull-back 
$\phi^*_{L_{\rm gen}}: z^q(\P^d_K|H_{\rm gen}, \bullet) \to
z^q(X_K|D_K, \bullet)$ is defined.
\item
$(\phi^*_{L_{\rm gen}} \circ \phi_{L_{\rm gen} *} \circ {\rm pr}^*_{K/k} - 
{\rm pr}^*_{K/k})$ maps 
$z^q_{\sW,e}(X|D, \bullet)$ to $z^q_{\sW_K, e-1}(X_K|D_K, \bullet)$.
\end{enumerate}
\end{lem}

\begin{proof}
Having established Lemmas~\ref{lem:Levine} and \ref{lem:Levine-birational-missing}, the proof of this lemma is identical to that of \cite[Lemma~3.5.6]{Levine}. The modulus condition plays no role in this deduction. Using Lemmas~\ref{lem:Levine} and \ref{lem:Levine-birational-missing} and the 
argument of \emph{loc.cit.} verbatim, one shows that given a cycle $\alpha \in z^q_{\sW, e}(X|D, p)$, there exists a dense open subset $\sU^{\alpha}_X \subseteq \Gr(N-d-1, H_m)$ such that for each $L \in  \sU^{\alpha}_X$, the linear projection away from $L$ defines a finite map $\phi_L: X \to \P^d_k$ satisfying the required conditions. 
This map is flat on $U$ as shown in the proof of
\lemref{lem:Levine-birational-missing}. Taking $L=L_{\rm gen}$ and
using \lemref{lem:fpb}, we get (1), (3) and (4). The map
$\phi_{L_{\rm gen} *}$ is defined by \cite[Proposition~2.10]{KP-1}.

The item (2) follows at once from our choice of $L_{\rm gen}$ and an elementary property of linear projection that a hyperplane section $X \cap H$ in $\P^N_k$ is a pull-back of a hyperplane of $\P^d_k$ via $\phi_L$ if and only if $L \subset H$. 
\end{proof}

We are now ready to prove our main theorem on the moving lemma for
the higher Chow groups of projective schemes with very ample modulus.

\begin{thm}\label{thm:Main-moving}
Let $k$ be any field and let $X$ be an equidimensional reduced 
projective scheme of dimension 
$d \ge 1$ over $k$. Let $D \subset X$ be a very ample effective Cartier 
divisor such that $X \setminus D$ is smooth over $k$. 
Let $\sW = \{W_1, \cdots , W_s\}$ be a finite collection of locally closed subsets of $X$ and let $e: \sW \to \Z_{\ge 0}$ be a set function. Then, the inclusion $z^q_{\sW, e-1}(X|D, \bullet) \inj z^q_{\sW,e}(X|D, \bullet)$ is a quasi-isomorphism. In particular, the inclusion $z^q_{\sW}(X|D, \bullet) \inj z^q(X|D, \bullet)$ is a quasi-isomorphism.
\end{thm}

\begin{proof}
The second part follows easily from the first part by induction because $z^q_{\sW}(X|D, \bullet) = z^q_{\sW,0}(X|D, \bullet)$ and $z^q(X|D, \bullet) = z^q_{\sW,q}(X|D, \bullet)$. We thus need to show that the quotient complex $\frac{z^q_{\sW,e}(X|D, \bullet)}{z^q_{\sW,e-1}(X|D, \bullet)}$ is acyclic.

First suppose that the theorem is true for all infinite fields and let $k$ be a finite field. Take a homology class $\alpha$ in this quotient. We choose two distinct primes $\ell_1$ and $\ell_2$, other than ${\rm char} (k)$, and take pro-$\ell_i$-extensions $\iota_i: \Spec (k_i) \to \Spec (k) $ for $i =1, 2$. Then the case of infinite fields tells us that $\iota^*_i(\alpha) = 0$ for $i = 1,2$. Hence, a descent argument implies that there are finite extensions $\tau_i: \Spec (k_i ') \to \Spec (k) $ of relatively prime degrees such that $\tau^*_i(\alpha) = 0$ for $i = 1,2$. Using the projection formula for
finite and flat morphisms (see \cite[Theorem~3.12]{KP-1}), this implies that $d_1\alpha = 0 = d_2\alpha$, where $(d_1, d_2) = 1$. We conclude that $\alpha = 0$.

We can now assume that $k$ is infinite. We set $\sW^0 = \{W_1 \setminus D, \cdots , W_s \setminus D\}$. Since a cycle in $z^q(X|D, p)$ does not intersect $D \times \square^p$, we see that $z^q_{\sW}(X|D, \bullet) = z^q_{\sW^0}(X|D, \bullet)$, and we may assume that $W \cap D = \emptyset$ for each $W \in \sW$.

Since $D$ is very ample, we can choose a closed embedding $\iota_X: X \inj \P^N_k$ and a hyperplane $H_m \subset \P^n_k$ such that $D = \iota^*(H_m)$. If $X = \P^N_k$, we are done by \thmref{thm:moving-PS}. So we can assume that $1 \le d \le N-1$.

It follows from \lemref{lem:Levine-box} that the map
\begin{equation}\label{eqn:Levine-box-0}
(\phi^*_{L_{\rm gen}} \circ \phi_{L_{\rm gen} *} \circ {\rm pr}^*_{K/k} - {\rm pr}^*_{K/k}):\frac{z^q_{\sW,e}(X|D, \bullet)}{z^q_{\sW,e-1}(X|D, \bullet)} \to  \frac{z^q_{\sW_K,e}(X_K|D_K, \bullet)}{z^q_{\sW_K,e-1}(X_K|D_K, \bullet)}
\end{equation}
is zero. On the other hand, each $\phi^*_{L_{\rm gen}} \circ \phi_{L_{\rm gen} *}$ factors as
\[
\frac{z^q_{\sW_K,e}(X_K|D_K, \bullet)}{z^q_{\sW_K,e-1}(X_K|D_K, \bullet)} \xrightarrow{\phi_{L_{\rm gen} *}} \frac{z^q_{\phi_{L_{\rm gen}}(\sW_K),e'}(\P^d_K|H_{\rm gen}, \bullet)}{z^q_{\phi_{L_{\rm gen}}(\sW_K),e'-1}(\P^d_K|H_{\rm gen}, \bullet)} \xrightarrow{\phi^*_{L_{\rm gen}}} \frac{z^q_{\sW_K,e}(X_K|D_K, \bullet)}{z^q_{\sW_K,e-1}(X_K|D_K, \bullet)}
\]
for some $e'$ (see \cite[\S~6C]{KP}). It follows from \corref{cor:moving-PS-Gen} that the middle complex is acyclic. This in turn implies that $\phi^*_{L_{\rm gen}} \circ \phi_{L_{\rm gen} *} = 0$ is zero on the level of homology. Combining this with ~\eqref{eqn:Levine-box-0}, we conclude that ${\rm pr}^*_{K/k}$ is zero on the level of homology. By \propref{prop:spread},  the complex $\frac{z^q_{\sW,e}(X|D, \bullet)}{z^q_{\sW,e-1}(X|D, \bullet)}$ is acyclic. This finishes the proof of the theorem.
\end{proof}

\section{Applications and remarks}\label{sec:sharp-0}
In this section, we apply our moving lemma to prove certain contravariant 
functoriality for higher Chow groups with modulus. We prove a
vanishing theorem on higher Chow groups with ample modulus.
We end the section by 
explaining why the very ampleness condition is crucial for proving the moving 
lemma.

\subsection{Contravariance}\label{sec:contraV}
Let $X$ be a quasi-projective scheme over a field $k$ and let 
$D \subset X$ be a very ample effective Cartier divisor. 
Recall from 
\cite[Theorem~3.12]{KP-1} if that $X$ is smooth, there is a cap product
$
\cap_X: \CH^q(X,p) \otimes_{\mathbb{Z}} \CH^{q'}(X|D,p') \to 
\CH^{q + q'}(X|D, p+p').
$
We prove the following contravariant functoriality for cycles with modulus. 

\begin{thm}\label{thm:PB-ample}
Let $f: Y \to X$ be a morphism of quasi-projective schemes over a field $k$, 
where $X$ is projective over $k$. Let $D \subset X$ be a very ample 
effective Cartier divisor such that $X \setminus D$ is smooth over $k$. 
Suppose that $f^*(D)$ is a Cartier divisor on $Y$
(i.e., no minimal or embedded component of $Y$ maps into $D$). 
Then, there exists a map
\[
f^*: z^q(X|D, \bullet) \to z^q(Y|f^*(D), \bullet)
\]
in the derived category
of abelian groups. In particular, there is a pull-back
$f^*: \CH^q(X|D, p) \to \CH^q(Y|f^*(D), p)$ for every $p,q \ge 0$. 

If $X$ and $Y$ are smooth and projective, then for every 
$a \in \CH^*(Y,\bullet)$ and $b \in \CH^*(X|D,\bullet)$, 
there is a projection formula $f_*(a \cap_Y f^*(b)) = f_*(a) \cap_X b$.
\end{thm}

\begin{proof}
The proof is a standard application of moving lemma for Chow groups. Set $E = f^*(D)$. For $0 \le i \le \dim(Y)$, let $X_i$ be the set of points $x \in X$ such that $\dim(f^{-1}(x)) \ge i$, where we assume $\dim(\emptyset) = -1$. Let $\sW$ be the collection of the irreducible components of all $X_i$. One checks that $\sW$ is a finite collection and the pull-back $f^*:z^q_{\sW}(X|D,\bullet) \to z^q(Y|E, \bullet)$ is defined (see \cite[Theorem~7.1]{KP}). 
We thus have maps $z^q (X|D, \bullet) \overset{q.iso}{\leftarrow} z^q _{\mathcal{W}} (X|D, \bullet) \overset{f^*}{\to} z^q (Y|E, \bullet)$ and 
\thmref{thm:Main-moving} says that the arrow on the left is a 
quasi-isomorphism. This proves the first part of the theorem.

To prove the projection formula, we can assume using \thmref{thm:Main-moving} that $b \in \CH^*(X|D,\bullet)$ is represented by a cycle $Z \in z^q_{\sW}(X|D,\bullet)$, where $\sW$ is as constructed above. By \cite[Lemma~3.10]{KP-1}, there is a finite collection of locally closed
subsets $\sC$ of $Y$ such that $Z' \boxtimes f^*(Z) \in z^q_{\Delta_Y}(Y|E,\bullet)$ for all $Z' \in z^q_{\sC}(Y,\bullet)$. By the moving lemma for Bloch's higher Chow groups, we can assume that $a \in \CH^*(Y,\bullet)$ is represented by a cycle $Z' \in z^q_{\sC}(Y,\bullet)$. In this case, it is straightforward to check that $f_*(Z') \boxtimes Z \in z^q_{\Delta_X}(X|D,\bullet)$ and $f_* \circ \Delta^*_Y(Z' \boxtimes f^*(Z)) = \Delta^*_X(f_*(Z') \boxtimes Z)$. This finishes the proof. 
\end{proof}

\begin{remk}\label{remk:PB-ample-0}
We remark that a pull-back map on higher Chow groups with modulus was constructed in \cite[Theorem~4.3]{KP-1}. But \thmref{thm:PB-ample}  
can not be deduced from \cite[Theorem~4.3]{KP-1}. The reason is that we make no assumption on the map $f$ while \emph{loc. cit.} assumes $D$ and $E$ to be the pull-backs of a divisor on a base scheme $S$ over which both $X$ and $Y$ should be smooth.

We also remark that \thmref{thm:PB-ample} proves a stronger statement than
giving a pull-back map on the higher Chow groups with modulus. 
This stronger version of \cite[Theorem~4.3]{KP-1} is not yet known.
\end{remk}

\begin{cor}\label{cor:PB-ample-1}
Let $r \ge 1$ be an integer and let $f: Y \to \P^r_k$ be a morphism of 
quasi-projective schemes over a field $k$. Let $D \subset \P^r_k$ be an 
effective Cartier divisor such that $f^*(D)$ is a Cartier divisor on $Y$. 
Then, there exists a pull-back $f^*: \CH^q(\P^r_k|D, p) \to \CH^q(Y|f^*(D), p)$ 
for every $p,q \ge 0$. 

If $Y$ is also smooth and projective, then for every $a \in \CH^*(Y,\bullet)$ 
and $b \in \CH^*(\P^r_k|D,\bullet)$, there is a projection formula 
$f_*(a \cap_Y f^*(b)) = f_*(a) \cap_X b$.
\end{cor}

\begin{proof}If $D=0$, then it is just an application of the moving lemma for 
usual higher Chow groups. If $D \not = 0$, then it is very ample so that 
Theorem \ref{thm:PB-ample} applies.
\end{proof}

\subsection{A vanishing theorem}\label{sec:vanishing}
The following result shows that the higher Chow groups of projective schemes 
(not necessarily smooth) 
with ample modulus are nontrivial only in high codimension. More precisely,

\begin{thm}\label{thm:small}
Let $X$ be a projective scheme of dimension $d \ge 1$ over a field $k$. 
Let $D \subset X$ be an ample effective Cartier divisor. Then, 
$z_s(X|D, p) = 0$ for 
$s > 0$. In particular, $\CH_s(X|D, p) = 0$ for $s > 0$.
\end{thm}

\begin{proof}
We can find a closed embedding $\iota_X: X \inj \P^N_k$ and a hyperplane 
$H \inj \P^N_k$ such that $nD = \iota^*_X(H)$ for some $n \gg 0$. 
Suppose $z_s(X|D, p) \not = 0$ for some $s \in \Z$. Let 
$\alpha \in z_s(X|D,p)$ be a nonzero admissible cycle and let $Z$ be an 
irreducible component of $\alpha$. Let 
${\rm pr}_{\P^N_k}: \P^N_k \times \square^p_k \to \P^N_k$ and 
${\rm pr}_{\square^p_k}: \P^N_k \times \square^p_k \to \square^p_k$ denote the 
projection maps. Let $y \in \square^p_k$ be any scheme point. For any map 
$W \to \square^p_k$, let $W_y$ denote the fiber $\Spec(k(y))  
\times_{\square^p_k} W$ over $y$. The modulus condition for $Z$ implies that 
$Z_y$ is a closed subscheme of $\P^N_y$ disjoint from $H_y$. In particular, 
$Z_y$ is a projective $k(y)$-scheme which is a closed subscheme of 
$(\P^N_y \setminus H_y) \simeq \A^N_{k(y)}$. Hence, it must be finite. We have 
thus shown that the projection map $Z \to \square^p_k$ is projective and 
quasi-finite, and hence finite. In other words, we must have 
$\dim(Z) = s+p \le p$, i.e., $s \le 0$. Thus $z_s(X|D, p ) = 0$ if $s > 0$, as 
desired.
\end{proof}

\subsection{Sharpness of the very ampleness condition}\label{sec:sharp}
We now show by an example that we can not weaken the very ampleness condition
to mere ampleness for the modulus divisor $D \subset X$. It also shows that the moving lemma for cycles with modulus on smooth affine schemes can not be proven using the method of linear projections, in general. 
This partly explains the need for the Nisnevich sheafification of the cycle
complex for the moving lemma of W. Kai \cite{Kai}.

Let $X$ be an elliptic curve over an algebraically closed field $k$ and let $D \subset X$ be a closed point. It is clear that $\sO_X(D)$ is ample. %(see \cite[Corollary~IV.3.3]{Hart}). 
We claim that there exists no pair $(f, D')$ consisting of a map $f:X \to \P^1_k$ and an effective Cartier divisor $D' \in {\rm Div}(\P^1_k)$ such that $D = f^*(D')$.

Suppose there does exist such a pair $(f, D')$. Observe that we must have $d := {\rm deg}(D') > 0$ and $D'$ is very ample. Let $\iota: \P^1_k \inj \P^d_k$ denote the closed embedding such that $\sO_{\P^1_k}(D') \simeq \iota^*(\sO_{\P^d_k}(1))$. This gives a regular map $\iota \circ f: X \to \P^d_k$ such that $(\iota \circ f)^*(\sO_{\P^d_k}(1)) = \sO_X(D)$. This implies that $\sO_X(D)$ is globally generated. However, by Riemann-Roch, one checks immediately that $h^0 (D) = 1$ in our case, i.e., ${\rm dim}(|D|) = 0$ and the unique element of $|D|$ vanishes at $D$, a contradiction.

Recall that the only technique yet available in the literature to prove the moving lemma for Bloch's higher Chow groups of smooth affine schemes is the method of \emph{linear projections}. Bloch proved the moving lemma for higher Chow groups of all smooth quasi-projective schemes (see \cite{Bl-2} and \cite[Proposition~2.5.2]{Bl-3}). But his proof depends on the moving lemma for smooth affine schemes proven in \cite{Bl-1} using linear projections.

Let us now consider the case of moving lemma for higher Chow groups with modulus on smooth affine schemes. Let $U$ be a smooth affine scheme over an algebraically closed field $k$ of characteristic zero. Let $D \subset U$ be a principal effective divisor $(u)$ such that the induced map $u: U \setminus D \to \A^1_k$ is smooth. We use the above example to show that even in this special case, the method of linear projections can not be used to prove the moving lemma for the higher Chow groups on $U$ with modulus $D$.  This makes proving the moving lemma for cycles with modulus on smooth affine or projective schemes very subtle and challenging.

Let $X$ be an elliptic curve over $k$ as above and let $D \inj X$ be a closed point. There exists an affine neighborhood $V \inj X$ of $D$ such that $D = (u)$ is principal on $V$. Let $u: V \to \A^1_k$ be the induced dominant map. We can find an affine neighborhood $U \inj V$ of $D$ such that $u: U \setminus D \to \A^1_k$ is {\'e}tale.

\begin{prop}\label{prop:sharpness-1}
There exists no pair $(f, D')$ consisting of a finite map $f:U \to \A^1_k$ and effective Cartier divisor $D' \inj \A^1_k$ such that $D = f^*(D')$.
\end{prop}

\begin{proof}
If such pair $(f, D')$ exists, then we get a commutative diagram
\begin{equation}\label{eqn:sharpness-2}
\xymatrix@C1pc{
U \ar[r]^{j'} \ar[d]_{f} & X \ar[d]^{f'} \\
\A^1_k \ar[r]^{j} & \P^1_k,}
\end{equation}
where the horizontal maps are open inclusions and vertical maps are finite. This finiteness implies that the above square is Cartesian. This in turn implies that we have a finite map $f': X' \to \P^1_k$ and effective Cartier divisor $D' \inj \P^1_k$ such that $D = f'^*(D')$ on $X$. But we have shown previously that this is not possible.
\end{proof}

\section{Higher Chow groups with modulus of a line bundle}\label{sec:H-inv}
Let $X$ be a quasi-projective scheme of dimension $d \ge 0$ over a field $k$. 
Let $f: \sL \to X$ be a line bundle and let $\iota: X \inj \sL$ be the 
$0$-section embedding. In this case, one knows that there is an isomorphism 
$\iota^*: \CH_*(\sL, \bullet) \xrightarrow{\simeq} \CH_*(X, \bullet)$ 
(up to a shift in dimension) of ordinary higher groups. 
Since the Chow groups with modulus are supposed to be the `relative
motivic cohomology' of the pair $(\sL, \iota(X))$, one expects 
$\CH_*(\sL|X, \bullet)$ to be trivial.

As an application of the moving techniques of \S~\ref{sec:PSM}, we show
in this section that every cycle in $z_s(\sL|X, \bullet)$ can be moved to a 
trivial cycle
so that this complex is acyclic. This gives an evidence in support of the
expectation that the Chow groups with modulus are the relative motivic
cohomology. It also provides examples where the higher Chow
groups of a variety with a modulus in an effective Cartier divisor are 
all zero. Note that this can never happen for the ordinary higher groups. 
The proof closely follows the arguments of 
Lemmas \ref{lem:Blow-up}, \ref{lem:Blow-up-Final}, and 
Proposition \ref{prop:moving-mod}.

Let $H: \sL \times \A^1_k \to \sL$ be the standard fiberwise contraction given explicitly as follows: for an affine open subset $U = \Spec(R) \subset X$ such that $f|_U$ is trivial, i.e., of the form $f|_U: U \times \A^1_k \to U$, write $\sL|_U = \Spec(R[t])$. Then, $H|_U: U \times \A^1_k \times \A^1_k \to U \times \A^1_k$ is induced by the polynomial map $R[x] \to R[t, x]$, given by $x \mapsto tx$.

For $n \ge 0$, let $H_n: \sL \times \A^1_k \times \ov{\square}^n_k \to \sL \times \ov{\square}^n_k$ be the map $H \times {\rm Id}_{\ov{\square}^n_k}$. For any irreducible closed admissible cycle $V \in z_s(\sL|X, n)$, let $H^*(V)$ denote the cycle associated to the flat pull-back $H_n^{-1}(V)$. Set $V' = (H^*(V))_{\rm red}$. We extend $H^*$ linearly to all cycles. Let $\ov{V} \inj \sL \times \ov{\square}^n_k$ denote the closure of $V$ and let $\nu_V: \ov{V}^N \to \sL \times \ov{\square}^n_k$ be the composition of the normalization and the inclusion. Let $\ov{V}'$ denote the closure of $V'$ in $\sL \times \ov{\square}^{n+1}_k$ and let $\nu_{V'}: {\ov{V}'}^N \to \sL \times \ov{\square}^{n+1}_k$ denote the map induced by the normalization of $\ov{V}'$.

\begin{lem}\label{lem:H-inv-1}
$V' \inj \sL \times {\square}^{n+1}_k$ has modulus $X$.
\end{lem}

\begin{proof}
Since the modulus condition is local on $\sL$, it is enough to show that $V' \cap (f^{-1}(U) \times {\square}^{n+1}_k)$ has modulus $U$ for every affine open subset $U \subset X$ over which $f$ is trivial. So we may assume $X = \Spec(R)$ is affine and $\sL = \Spec(R[X])$ is trivial. In this case, $H: U \times \A^1_k \times \A^1_k \to U \times \A^1_k$ is given by $H(u, x, y) = (u, xy)$. Since $U$ plays no role in this map, we can drop it and assume $U = \Spec(k)$ so that $H: \A^1_k \times \A^1_k \to \A^1_k$ is the multiplication map. This map uniquely extends to a rational map $H: \P^1_k \times \P^1_k \dashrightarrow \P^1_k$, given by $H\left((X_0;X_1), (T_0;T_1)\right) = (X_0T_0; X_1T_1)$, 
which is regular on $W =(\P^1_k \times \P^1_k) \setminus 
\{(0, \infty), (\infty, 0)\}$.

We next observe that since the modulus divisor is $U = \{0\} \inj \A^1_k$, to check the modulus condition for $H^{-1}(V)$ is equivalent to check the modulus 
$(\{0\} \times \A^1_k)$ for $(H|_{W \times {\square}^{n}_k})^{-1}(V_1)$, where $V_1$ is the closure of $V$ in $\P^1_k \times {\square}^{n}_k$. We can thus replace $\A^1_k$ by $\P^1_k$ as the target space of $H$ and $\ov{V}'$ by its closure in
$\P^1_k \times \ov{\square}^{n+1}_k$ in order to check the modulus condition for 
$V'$.

Let $\pi: \Gamma \to \P^1_k \times \P^1_k$ be the blow-up along 
$\Sigma = \{(0,\infty), (\infty, 0)\}$. 
It is easily checked (see the proof of \lemref{lem:Blow-up}) that 
$\Gamma \inj \P^1_k \times \P^1_k \times \P^1_k$ is the closed subscheme given 
by $\Gamma = \{\left((X_0;X_1), (T_0;T_1), (Y_1;Y_0)\right)| 
X_0T_0Y_0 = X_1T_1Y_1\}$. 
Define a map $\ov{H}: \Gamma \to \P^1_k$ by 
$\ov{H}\left((X_0;X_1), (T_0;T_1), (Y_1;Y_0)\right) = (Y_1;Y_0)$.

We claim that $\ov{H}|_W = H$. 
To check this, let $U_1 = \{\left((X_0;X_1), (T_0;T_1)\right)|
X_1 \neq 0 \neq T_0\}$ and $U_2 = \{\left((X_0;X_1), (T_0;T_1)\right)|
X_0 \neq 0 \neq T_1\}$ be two open subsets of $\P^1_k \times \P^1_k$.
In the affine coordinates $(x_0, t_1) \in U_1 \simeq \A^2_k$, 
the restriction of $H$ on $U_1 \cap W$ is given by
$H(x_0, t_1) = (x_0;t_1)$ and the restriction of $\ov{H}$ on 
$\pi^{-1}(U_1) \cap W \cap (x_0 \neq 0)$ is given by 
$\ov{H}\left((x_0,t_1, (1; x^{-1}_0t_1)\right) = (1; x^{-1}_0t_1)
= (x_0; t_1) = H(x_0, t_1)$. The restriction of $\ov{H}$ on 
$\pi^{-1}(U_1) \cap W \cap (t_1 \neq 0)$ is given by 
$\ov{H}\left((x_0,t_1, (x_0t^{-1}_1; 1)\right) = (x_0t^{-1}_1; 1)
= (x_0; t_1) = H(x_0, t_1)$.

The restriction of $H$ on $U_2 \cap W$ is given by
$H(x_1, t_0) = (t_0;x_1)$ and the restriction of $\ov{H}$ on 
$\pi^{-1}(U_2) \cap W \cap (x_1 \neq 0)$ is given by 
$\ov{H}\left((x_1,t_0, (x^{-1}_1t_0;1)\right) = (x^{-1}_1t_0;1) =
(t_0; x_1)= H(x_1, t_0)$. The restriction of $\ov{H}$ on 
$\pi^{-1}(U_1) \cap W \cap (t_0 \neq 0)$ is given by 
$\ov{H}\left((x_1,t_0, (1;x_1t^{-1}_0)\right) = (1; x_1t^{-1}_0)
= (t_0; x_1) = H(x_1, t_0)$. Since $\pi$ is an isomorphism away from
$U_1 \cup U_2$, we have shown that $\ov{H}|_W = H$.  

It follows from the claim that there is a commutative diagram
\begin{equation}\label{eqn:Blow-up-Local}
\xymatrix@C1pc{
\pi^{-1}(W) \ar@{^{(}->}[r]^>>>>{j_1} \ar[d]_{\simeq} & \Gamma 
\ar@{->>}[d]_{\pi} \ar[dr]^{\ov{H}} & \\
W \ar@{^{(}->} [r]^>>>>{j} \ar@/_1.5pc/[rr] & \P^1_k \times \P^1_k
\ar@{.>}[r]^{\ \ H} & \P^1_k.}
\end{equation}

Let $E = \pi^*((0, \infty))$ denote one of the two components of the exceptional
divisor for $\pi$ and 
let $D = U  = \{0\} \inj \P^1_k$. We have 
$\pi^*(D \times \P^1_k) = (D \times \P^1_k)+ E$. Similarly, we have 
$\pi^*(\P^1_k \times \{\infty\}) = (\P^1_k \times \{\infty\}) +E$ in 
${\rm Div}(\Gamma)$. 
Set $E_{n} = E \times \ov{\square}^n_k$. 

Let $Z \inj \Gamma \times \ov{\square}^{n}_k$ denote the strict transform of $\ov{V}'$. Since $\ov{H}_{n}(Z \cap (\pi^{-1}(W) \times {\square}^{n}_k)) = V$ and since $\ov{H}_{n}$ is projective, we must have $\ov{H}_{n}(Z) = \ov{V}$.
We remark at this stage that ensuring the projectivity of $\ov{H}_n$
was the reason for us to replace $\A^1_k \times \A^1_k$ 
by $\P^1_k \times \P^1_k$ and $\A^1_k$ by $\P^1_k$ as the source and the target
of $H$.

We now have a commutative diagram
\begin{equation}\label{eqn:generic-pt-3-Local}
\xymatrix@C1pc{
Z^N \ar[dr]^{\nu_Z} \ar[rr]^{f} \ar[dd]_{g} & & \ov{V}^N \ar[d]^{\nu_V} \\
& \Gamma \times \ov{\square}^{n}_k \ar[r]^{\ov{H}_{n}} \ar[d]^{\pi_{n}}
& \P^1_k \times \ov{\square}^{n}_k \\
{\ov{V}'}^N \ar[r]_<<<{\nu_{V'}} & \P^1_k \times \ov{\square}^{n+1}_k, & }
\end{equation}
where $f$ and $g$ are the unique maps induced by the universal property of normalization for dominant maps. Since $f$ is a surjective map of integral schemes, the modulus condition for $V$ implies that $(\nu_V \circ f)^*(\P^1_k \times F^{\infty}_n) \ge (\nu_V \circ f)^*(D \times \ov{\square}^n_k)$ on $Z^N$. In particular, we get $(\ov{H}_{n} \circ \nu_Z)^*(\P^1_k \times F^{\infty}_n) \ge (\ov{H}_n \circ \nu_Z)^*(D \times \ov{\square}^n_k)$ on $Z^N$. Equivalently, we have
\begin{equation}\label{eqn:generic-pt-4-Local}
\nu^*_Z(\Gamma \times F^{\infty}_n) \ge \nu^*_Z(\ov{H}^*(D) \times \ov{\square}^n_K).
\end{equation}

Since $H^*(D) = (\A^1_k \times \{0\}) + (\{0\} \times \square_k)$, we get ${j}^*_{1,n} \circ \ov{H}^*_{n}(D \times \ov{\square}^n_k) = {j}^*_{1,n}(\A^1_k \times F^{0}_{n,n+1}) + {j}^*_{1,n}(D \times \ov{\square}^{n+1}_k)$, where $j_1: W \inj \Gamma$ is the inclusion. Since $\A^1_k \times F^{0}_{n,n+1}$ and $D \times  \ov{\square}^{n+1}_k$ are irreducible, we get $\ov{H}^*(D) \times \ov{\square}^n_k \ge 
(\P^1_k \times F^{0}_{n, n+1}) + (D \times \ov{\square}^{n+1}_k)$ on $\Gamma \times \ov{\square}^n_k$. Combining this with ~\eqref{eqn:generic-pt-4-Local}, we get
\begin{equation}\label{eqn:generic-pt-4-0-Local}
\nu^*_Z(\Gamma \times F^{\infty}_n) \ge
%\nu^*_Z(\A^1_k \times F^{0}_{n,n+1}) + 
\nu^*_Z(D \times \ov{\square}^{n+1}_k). 
\end{equation}

This in turn implies that 
\[
\begin{array}{lll}
(\pi_{n} \circ \nu_Z)^*(\P^1_k \times F^{\infty}_{n+1}) & = &
(\pi_{n} \circ \nu_Z)^*(\P^1_k \times F^{\infty}_n \times \ov{\square}_k)  \\
& &
+ (\pi_{n} \circ \nu_Z)^*(\P^1_k \times \ov{\square}^n_k \times \{\infty\}) \\
& = & \nu^*_Z(\Gamma \times F^{\infty}_n)
+ (\pi_{n} \circ \nu_Z)^*(\P^1_k \times \ov{\square}^n_k \times \{\infty\}) \\
& \ge & \nu^*_Z(D \times \ov{\square}^{n+1}_k) 
+ (\pi_{n} \circ \nu_Z)^*(\P^1_k \times \ov{\square}^n_k \times \{\infty\}) \\
& {=} & 
\nu^*_Z(D \times \ov{\square}^{n+1}_k) 
+ \nu^*_Z(E_{n}) + \nu^*_Z(\P^1_k \times \ov{\square}^n_k \times \{\infty\}) \\
& {=} & 
(\pi_{n} \circ \nu_Z)^*(D \times \ov{\square}^{n+1}_k) +
\nu^*_Z(\P^1_k \times \ov{\square}^n_k \times \{\infty\}) \\
& \ge & (\pi_{n} \circ \nu_Z)^*(D \times \ov{\square}^{n+1}_k). \\
\end{array}
\]

Using ~\eqref{eqn:generic-pt-3-Local}, this gives 
$g^*(\nu^*_{V'}(\P^1_k \times F^{\infty}_{n+1})) \ge  
g^*(\nu^*_{V'}(D \times \ov{\square}^{n+1}_k))$. 
We now apply \lemref{lem:cancel} to conclude that 
$\nu^*_{V'}(\P^1_k \times F^{\infty}_{n+1}) \ge 
\nu^*_{V'}(D \times \ov{\square}^{n+1}_k)$ and this is the modulus condition 
for $V'$.
\end{proof}

\begin{lem}\label{lem:H-inv-2}
$V' \inj \sL \times \square^{n+1}_k$ intersects with all faces properly.
\end{lem}

\begin{proof}
Since $H$ is flat, $V'$ intersects properly with all faces of $\square^{n+1}_k$ of the form $F \times \square_k$. Since $\iota^*_{n+1, n+1, 1}(V') = V$ which intersects faces of $\square^n_k$ properly, we see that $V'$ intersects $F^1_{n+1, n+1}$ properly. Since $V \cap (X \times \square^n_k) = \emptyset$, we must have $\iota^*_{n+1, n+1, 0}(V') = 0$. We have thus shown that $V'$ satisfies the face condition. 
\end{proof}

\begin{thm}\label{thm:H-inv-main}
Let $X$ be a quasi-projective scheme over a field $k$ and let
$f: \sL \to X$ be a line bundle. Let $\iota:X \inj \sL$ denote the
0-section embedding. 
Then, the cycle complex $z_s(\sL|X, \bullet)$ is acyclic for all $s \in \Z$.
\end{thm}

\begin{proof}
It follows from Lemmas~\ref{lem:H-inv-1} and ~\ref{lem:H-inv-2} that $H: \sL \times \A^1_k \to \sL$ defines a chain homotopy $H^*: z_s(\sL|X, \bullet) \to 
z_s(\sL|X, \bullet)[-1]$ between $H^*_{0} = (H|_{\sL \times 0})^*$ and 
$H^*_1 = (H|_{\sL \times 1})^*$. It is clear that 
$H^*_1 = {\rm Id}_{z_s(\sL|X, \bullet)}$ and the modulus condition implies that 
$H^*_{0} = 0$. It follows that $z_s(\sL|X, \bullet)$ is acyclic.
\end{proof}

\noindent\emph{Acknowledgments.} JP was partially supported by the National Research Foundation of Korea (NRF) grant No. 2015R1A2A2A01004120 funded by the Korean government (MSIP). AK was partially supported by the Swarna Jayanti Fellowship, 2011. The authors are deeply indebted to the referee, who
so thoroughly read the paper and suggested many valuable corrections and
simplifications.


\begin{thebibliography}{99}
\bibitem{BS} F. Binda, S. Saito, {\sl Relative cycles with moduli and regulator maps\/}, arXiv:1412.0385, (2014).

\bibitem{Bl-1} S. Bloch, {\sl Algebraic cycles and higher $K$-theory\/}, Adv. Math., \textbf{61}, (1986), 267--304.

\bibitem{Bl-2} S. Bloch, {\sl The moving lemma for higher Chow groups\/}, J. Algebraic Geom., \textbf{3}, (1994), no. 3, 537--568.

\bibitem{Bl-3} S. Bloch, {\sl Basic definition and properties of higher Chow groups\/}, web page of Spencer Bloch, University of Chicago, (2005).

\bibitem{BE2} S. Bloch, H. Esnault, {\sl The additive dilogarithm\/}, Doc. Math., \textbf{Extra Vol.}, (2003), 131--155.

\bibitem{Hart} R. Hartshorne, {\sl Algebraic Geometry\/},  Graduate Texts in Math., {\bf 52}, Springer-Verlag, New York., 1977.

\bibitem{Kai} W. Kai, {\sl A moving lemma for algebraic cycles with modulus and contravariance\/}, arXiv:1507.07619v1, (2015). 

\bibitem{KS} M. Kerz, S. Saito, {\sl Chow group of $0$-cycles with modulus and higher dimensional class field theory\/}, Duke Math. J. (to appear), 
DOI 10.1215/00127094-3644902, arXiv:1304.4400v1, (2013).

%\bibitem{Krish} A. Krishna, {\sl $0$-cycles with modulus on surfaces\/}, 
%Algebra \& Number Theory, {\bf 9}, (2015), 2397--2415.  


\bibitem{KL} A. Krishna, M. Levine, {\sl Additive higher Chow groups of schemes\/}, J. Reine Angew. Math., \textbf{619}, (2008), 75--140.

\bibitem{KP} A. Krishna, J. Park, {\sl Moving lemma for additive higher Chow groups\/}, Algebra \& Number Theory, \textbf{6}, (2012), no. 2, 293--326.

%\bibitem{KP Jussieu} A. Krishna, J. Park, {\sl Mixed motives over $k[t]/ (t^{m+1})$\/}, J. Inst. Math. Jussieu, \textbf{11}, (2012), no. 3, 611--657.

%\bibitem{KP-3} A. Krishna, J. Park, {\sl DGA-structure on additive higher Chow groups}, Int. Math. Res. Not. \textbf{2015}, no. 1, 1--54.

\bibitem{KP-1} A. Krishna, J. Park, {\sl A module structure and a vanishing theorem for cycles with modulus\/}, Math. Res. Lett. (to appear), 
arXiv:1412.7396v3, (2015).

\bibitem{KP-2} A. Krishna, J. Park, {\sl On the additive higher Chow groups of affine schemes\/}, Doc. Math. 21 (2016) 49-89.

%\bibitem{KPcrys} A. Krishna, J. Park, {\sl Algebraic cycles and crystalline cohomology\/}, arXiv:1504.08181v3, (2015).

\bibitem{Levine} M. Levine, {\sl Mixed Motives\/}, Mathematical Surveys and Monographs, {\bf 57}, American Mathematical Society, Providence, RI, (1998).

\bibitem{P1} J. Park, {\sl Regulators on additive higher Chow groups\/}, Amer. J. Math., \textbf{131}, (2009), no. 1, 257--276.

%\bibitem{R} K. R\"ulling, {\sl The generalized de Rham-Witt complex over a field is a complex of zero-cycles\/}, J. Algebraic Geom., \textbf{16}, (2007), no. 1, 109--169.


\end{thebibliography}
\end{document}